\documentclass[12pt,centertags]{amsart}

\usepackage{amsmath,amstext,amsthm,a4,amssymb,amscd,mathrsfs,amsfonts}
\usepackage{mathtools}

\usepackage{enumitem}
\usepackage[mathscr]{eucal}
\usepackage{epsf}

\usepackage{graphicx}
\usepackage{caption}

\usepackage{tikz-cd}
\usetikzlibrary{cd}

\usepackage{tikz}
\usepackage{pgfplots} 
\pgfplotsset{compat=1.18}

\numberwithin{equation}{section}

\newcommand{\field}[1]{\mathbb{#1}}
\newcommand{\Z}{\field{Z}}
\newcommand{\R}{\field{R}}
\newcommand{\C}{\field{C}}
\renewcommand{\H}{\field{H}}
\newcommand{\N}{\field{N}}

\newcommand{\p}{\partial}

\DeclareMathAlphabet{\mathscr}{U}{rsfs}{m}{n}

\def\cS{\mathscr{S}}

\def\mE{\mathcal{E}}
\def\mF{\mathcal{F}}
\def\mH{\mathcal{H}}

\def\mL{\mathcal{L}}

\def\mR{\mathcal{R}}

\newcommand\mS{\mathcal{S}}

\def\CP{\C P}
\def\HP{\H P}

\DeclareMathOperator{\End}{End}

\DeclareMathOperator{\Ker}{Ker}
\DeclareMathOperator{\Coker}{Coker}
\DeclareMathOperator{\Image}{Im}

\DeclareMathOperator{\Id}{Id}

\DeclareMathOperator{\Ind}{Ind}

\DeclareMathOperator{\vol}{vol}
\DeclareMathOperator{\ch}{ch}

\DeclareMathOperator{\pr}{pr}
\DeclareMathOperator{\dd}{d}
\DeclareMathOperator{\Cl}{Cl}
\DeclareMathOperator{\grad}{grad}

\newtheorem{thm}{Theorem}[section]
\newtheorem{lemma}[thm]{Lemma}
\newtheorem{prop}[thm]{Proposition}
\newtheorem{cor}[thm]{Corollary}
\theoremstyle{definition}
\newtheorem{rmk}[thm]{Remark}
\theoremstyle{definition}
\newtheorem{ex}[thm]{Example}
\theoremstyle{definition}
\newtheorem{defn}[thm]{Definition}
\theoremstyle{definition}
\newtheorem{conj}[thm]{Conjecture}

\newcommand{\be}{\begin{eqnarray}}
\newcommand{\ee}{\end{eqnarray}}
\newcommand{\ov}{\overline}

\newcommand{\comment}[1]{}

\begin{document}
	
\title{A Product Formula for Family Indices and Family Band Width Estimates}

\author{Chenkai Song}
\address{Chern Institute of Mathematics \& LPMC, Nankai
University, Tianjin 300071, P.R. China}
\email{songck@mail.nankai.edu.cn}

\begin{abstract}
	We extend Gromov's conjecture on the sharp width estimate for Riemannian bands with positive scalar curvature to the family case and prove that it holds for fiber bundles with infinite family $\widehat{A}$-area. The method we employ is based on Dirac operators and the family index theory. Our proof relies on a product formula for index bundles established in this paper.
\end{abstract}

\maketitle


\setcounter{section}{-1}
\section{Introduction}
Manifolds of positive scalar curvature (PSC) have been an important topic in differential geometry. A basic observation is that if a closed manifold $Z$ admits a PSC metric, then the product manifold $Z\times\R$ also admits such a complete metric. In contrast, if $Z$ admits no metric of PSC, then although PSC metrics exist on products $Z\times I$ for short intervals $I\subset\R$ \cite[Section 2]{Gro2018}, one may not always be able to extend such metrics to arbitrarily long bands. Gromov formulated this observation quantitatively in the following conjecture.

\begin{conj}\label{BandWidthConj}{\cite[Section 11.12 Conjecture C]{Gro2018}}		
	Let $Z$ be a connected closed manifold that admits no PSC metric. Denote $n-1\coloneq\dim Z\geqslant5$. If $g$ is a metric on the Riemannian band $Z\times[-1, 1]$ with scalar curvature bounded from below by $\sigma>0$, then
    \begin{equation}\label{BandWithEstimateSingle}
        \mathrm{dist}_g\{Z\times\{-1\},Z\times\{1\}\}\leqslant2\pi\sqrt{\frac{n-1}{n\sigma}}.
    \end{equation}
\end{conj}

The condition $\dim Z\geqslant 5$ is imposed because the case $\dim Z=4$ admits counterexamples based on Seiberg--Witten invariants. See \cite[Counterexample 4.16]{RS2014}. For $\dim Z\leqslant 3$, the conjecture is known to hold by the results of \cite{Rade2023} and \cite{LiZhang2025}.

Gromov showed that the constant in the conjecture is sharp and verified Conjecture~\ref{BandWidthConj} for overtorical bands using geometric measure theory \cite{Gro2018}. In recent years, the Dirac operators and index theory have turned out to be useful tools for this problem. The conjecture has been established for $Z$ with non-vanishing Rosenberg index \cite{Zeid2022,Ce2020,Zeid2020}, for odd-dimensional bands with infinite vertical $\widehat{A}$-area \cite{CZ2024} and for $Z$ with non-vanishing higher index \cite{GXY2020,Yu2024}. The method of $\mu$-bubbles has also been applied to Conjecture \ref{BandWidthConj}. See, for instance, \cite[Section 5]{Gro2023} and \cite{Zhu2021,Rade2023}.

In this paper, inspired by Atiyah and Singer's family index theorem \cite{ASIV}, we consider a family of the closed manifold $Z$, namely a fiber bundle $\pi\colon N\to B$ with fiber $Z$. Our focus is on fiber bundles $N$ that admit no fiberwise PSC metric. This viewpoint leads to new obstructions which do not arise in the single-manifold setting. We state the following family version of Gromov's band width conjecture.

\begin{conj}\label{FamilyBandWidth}
	Let $N\to B$ be a fiber bundle with base $B$ a connected closed manifold, and with all fibers diffeomorphic to a fixed connected closed manifold $Z$. Suppose there is no fiberwise PSC metric on $N$. Denote $n-1\coloneq\dim\,Z\geqslant2$, with $n-1\neq 4$. If the fiber bundle $N\times[-1,1]\to B$ admits a metric $g$ with fiberwise scalar curvature bounded from below by $\sigma>0$ independent of $b\in B$, then
    \begin{equation}\label{BandWithEstimateFamily}
        \mathrm{dist}_g\{N\times\{-1\},N\times\{1\}\}\leqslant2\pi\sqrt{\frac{n-1}{n\sigma}}.
    \end{equation}
\end{conj}

\begin{rmk}\label{RelationOfConjs}
	This conjecture reduces to Conjecture \ref{BandWidthConj} when the base space $B$ is a point. In particular, the counterexamples to Conjecture \ref{BandWidthConj} in the case $\dim Z=4$ also apply to this family version.
\end{rmk}

The following example, obtained by a slight modification of Gromov's construction in \cite[Section 2]{Gro2018}, shows that there always exist fiberwise PSC metrics on $N\times[-1,1]$ whose band width is sufficiently small. Moreover, this example demonstrates that the constant $2\pi\sqrt{\frac{n-1}{n}}$ in Conjecture \ref{FamilyBandWidth} is sharp.

\begin{ex}
	Fix a Riemannian metric $g^V$ on the vertical tangent bundle $T^VN$ of $N$. Gromov showed that for a sufficiently large constant $c>0$, the metric
	\begin{equation}
		\bar{g}^V\coloneq c^2\varphi(t)g^V\oplus\dd t\otimes\dd t
	\end{equation}
	on the vertical tangent bundle $T^V\big(N\times\big[-\big(\frac{\pi}{n}-\frac{\varepsilon}{2}\big),\frac{\pi}{n}-\frac{\varepsilon}{2}\big]\big)$ restricts, on each fiber, to a Riemannian metric whose scalar curvature is bounded from below by $n(n-1)$. Here $t$ denotes the coordinate on the interval, and $\varphi$ is the function constructed in \cite[Page 654]{Gro2018} for a constant $\sigma$ slightly larger than $n(n-1)$.

	Next, choose a horizontal tangent bundle $T^HN$ on $N$ and equip it with an arbitrary Riemannian metric. Let $\bar{g}^H$ denote the pullback of this metric to the horizontal tangent bundle $T^H\big(N\times\big[-\big(\frac{\pi}{n}-\frac{\varepsilon}{2}\big),\frac{\pi}{n}-\frac{\varepsilon}{2}\big]\big)$ via the projection $N\times\big[-\big(\frac{\pi}{n}-\frac{\varepsilon}{2}\big),\frac{\pi}{n}-\frac{\varepsilon}{2}\big]\to N$. Then the total metric $\bar{g}^V\oplus\bar{g}^H$ defines a Riemannian metric on $N\times\big[-(\frac{\pi}{n}-\frac{\varepsilon}{2}),\,\frac{\pi}{n}-\frac{\varepsilon}{2}\big]$ with fiberwise scalar curvature bounded from below by $n(n-1)$. The distance between the two boundary components with respect to this metric equals $\frac{2\pi}{n}-\varepsilon$.
\end{ex}

If the fiber $Z$ satisfies the assumptions of any known result of Conjecture~\ref{BandWidthConj} (e.g. $Z$ has non-vanishing higher index), then the corresponding band width estimate holds on each fiber. The family estimate then follows immediately from the pointwise one and the inequality
\begin{equation}
	\mathrm{dist}_g\{N\times\{-1\},N\times\{1\}\}\leqslant \inf_{b\in B}\mathrm{dist}_{g_b}\{Z_b\times\{-1\},Z_b\times\{1\}\},
\end{equation}
where $Z_b\coloneq\pi^{-1}(b)$ and $g_b$ is the restriction of $g$ to $Z_b\times[-1,1]$.

Similarly, if the total space $N$ itself satisfies the assumptions of any known result on Conjecture~\ref{BandWidthConj}, then the estimate \eqref{BandWithEstimateFamily}, with a non-optimal constant, follows by applying those results directly to $N$. Indeed, by \cite[(9.37), Lemma~9.69]{Besse1987}, one may rescale the horizontal part of the metric on $N\times[-1,1]$ by a sufficiently large constant to obtain a PSC metric (cf.\ \cite[Definition~2.1]{HSS2014}), and the band width is non-decreasing under this deformation. Hence, no essentially new phenomenon occurs in these cases.

However, the situation changes fundamentally when the fiber $Z$ or the total space $N$ admits a PSC metric. Indeed, there exist fiber bundles $\pi\colon N \to B$ whose fibers or total spaces admit PSC metrics, while $N$ itself admits no fiberwise PSC metric. Examples in which the fiber $Z$ admits a PSC metric can be found in \cite[Theorem~1.4 and Remark~1.5]{HSS2014}, \cite[Theorem]{KKR2021}, and \cite[Theorem~3.2]{BE2023}. For instance, Krannich, Kupers and Randal-Williams constructed a bundle $E \to S^4$ whose total space has non-vanishing $\widehat{A}$-genus and whose fiber is $\HP^2$, a manifold known to admit a PSC metric.

Now consider the fiber bundle $E\times S^2\to S^4\times S^2$, obtained by pulling back the above bundle $E$ via the projection $S^4\times S^2\to S^4$. This bundle still admits no fiberwise PSC metric, whereas its total space $E\times S^2$ admits a PSC metric. We also construct another example of a fiber bundle that is not a product, whose total space, base, and individual fibers all admit PSC metrics, but which admits no fiberwise PSC metric (see Example~\ref{InfiniteAhatAeraExample2}).

For these examples, the width estimate \eqref{BandWithEstimateFamily} cannot be deduced from previously known results in the case of a single band.

The purpose of this paper is to show that, even for such fiber bundles, the band width estimate \eqref{BandWithEstimateFamily} continues to hold. More precisely, our result implies that even if the fiber $Z$ admits a PSC metric, any fiberwise PSC metric on $N\times[-1,1]$ satisfies \eqref{BandWithEstimateFamily} whenever the fiber bundle $N$ is obtained by a sufficiently complicated twisting of $Z$. The same estimate also holds when the total space $N$ itself admits a PSC metric, but has obstructions arising from the bundle structure, such as the nonvanishing of family indices.

Our proof relies on family index theory and is carried out under the assumption that the vertical tangent bundle of $\pi$ is spin.\footnote{All examples mentioned above satisfy this assumption.} From this perspective, the obstruction to fiberwise PSC metrics arises from a condition analogous to a non-vanishing family index, which is strictly weaker than the non-vanishing of the index of the fiber or the total space.

A key technique in our proof is a product formula for family indices. This formula can be viewed as a partial extension of the partitioned manifold index theorem \cite[Theorem~1.5]{Ang19932} or \cite[Theorem~1]{Rade1994} (see also \cite[Theorem~A.1]{Zeid2022}) to the family setting. It identifies the index bundle of a family of twisted Dirac operators on $N \times \R \to B$, deformed by a zeroth-order potential term, with the index bundle of the corresponding family on $N \to B$ (see Theorem~\ref{FamilyProductFormula2} and Theorem~\ref{FamilyProductFormula2'}). To establish the invertibility of the deformed Dirac operators, we choose to adapt analytic techniques developed by Wang--Xie--Yu in \cite{Yu2024}.

To state our main result in a more general framework, we introduce the notion of fiber bundles with infinite family $\widehat{A}$-area, which is closely related to \cite[Definition 7.3]{CZ2024}. Let $\pi\colon N\to B$ be a fiber bundle with connected closed base $B$ and connected closed fiber $Z$, and assume that the vertical tangent bundle $T^VN$ is oriented and spin. Recall that given an Hermitian vector bundle $\mE\to N$ with an Hermitian connection $\nabla^{\mE}$, the associated family of twisted Dirac operators determines an index bundle in $K^{\dim Z}(B)$ \cite{AS1969,ASIV}. We say the bundle $N\to B$ has infinite family $\widehat{A}$-area (see Definition \ref{InfiniteAhatArea}) if for every $\varepsilon>0$, there exists an Hermitian vector bundle $(\mE,\nabla^{\mE})$ over $N$ such that:
\begin{itemize}
	\item[$(a)$] The inequality $\|R^{E_b}\|_{L^\infty} \leq \varepsilon$ holds for all $b \in B$, where $\big(E_b,\nabla^{E_b}\big)$ is the natural restriction of $\big(\mE,\nabla^{\mE}\big)$ to $Z_b\coloneq\pi^{-1}(b)$ with $R^{E_b}\coloneq\big(\nabla^{E_b}\big)^2$ being the curvature;
	\item[$(b)$] The index bundle in $K^{\dim Z}(B)$ of the corresponding family of twisted Dirac operators does not vanish.
\end{itemize}

The examples discussed above all belong to this class of fiber bundles (see Example \ref{InfiniteAhatAeraExample} and Example \ref{InfiniteAhatAeraExample2}).

\begin{rmk}
	One could also define a cohomological version of infinite family $\widehat{A}$-area for a fiber bundle $N\to B$ whose vertical tangent bundle $T^VN$ is only assumed to be oriented, by replacing the nonvanishing condition in $K$-theory in $(b)$ by
	$$\int_{N/B}\widehat{A}\left(T^VN\right)\ch(\mE)\neq0\in H^*(B,\C).$$
	However, this condition only detects the Chern character of the index bundle and hence loses the torsion information in $K^{\dim Z}(B)$. Moreover, the spin condition is essential for the Dirac operator method used in this paper. For these reasons, we impose the spin condition on $T^VN$ and formulate the obstruction at the level of $K$-theory.
\end{rmk}

We can now state our main theorem, which solves Conjecture \ref{FamilyBandWidth} for fiber bundles with infinite family $\widehat{A}$-area.

\begin{thm}\label{MainTheorem}
	Let $\pi\colon N\to B$ be a fiber bundle with fiber $Z$. Denote the fiber over $b\in B$ by $Z_b$. Assume that $N$ has infinite family $\widehat{A}$-area. In particular, the vertical tangent bundle $T^VN$ is oriented and spin. Consider $M\coloneq N\times[-1,1]$, which naturally carries a fiber bundle structure over $B$ with fiber $Z\times[-1,1]$.
	
	Let $g$ be a Riemannian metric on $M$, and denote by $g_b$ its restriction to the fiber $Z_b\times[-1,1]$. Suppose the corresponding scalar curvature $k_{g_b}$ admits a lower bound $\sigma>0$ independent of $b\in B$. Then
	\begin{equation}
		\mathrm{dist}_{g}\left\{N\times\{-1\},N\times\{1\}\right\}\leqslant2\pi\sqrt{\frac{n-1}{n\sigma}},
	\end{equation}
	where $n\coloneq\dim Z+1$.
\end{thm}

The rest of the paper is organized as follows. In Subsection \ref{Subsection1.1}, we review the framework of the family index theorem and consider the case of noncompact fibers in a special setting. Subsection \ref{Subsection1.2} recalls the definitions of twisted spinor bundles and families of twisted Dirac operators. Section \ref{Section2} presents the statement and proof of the family product formula, treating even-dimensional and odd-dimensional fibers separately. In Section \ref{Section3}, we introduce the definition of fiber bundles with infinite family $\widehat{A}$-area and prove Theorem \ref{MainTheorem}. The final section is an appendix in which we prove two lemmas, one characterizing when two index bundles in the $K^0$ group coincide and the other for $K^1$ group. These two lemmas are used in the proof of the family product formula.

\section{Index Bundle and Family of twisted Dirac Operators}\label{Section1}
In this section, we begin by reviewing the preliminaries of the family index theory, with a particular focus on the case where the fibers are nearly of the form $Z\times\R$ for a closed manifold $Z$. We then recall the construction of the family of twisted Dirac operators on a fiber bundle with spin vertical tangent bundle.

\subsection{Index Bundle for Noncompact Fibers}\label{Subsection1.1}
Let $\pi\colon N\to B$ be a fiber bundle with fiber $Z$ a connected manifold without boundary. Throughout this paper, we assume that the base manifold $B$ is connected and closed. For each $b\in B$, denote the fiber of $\pi$ over $b$ by $Z_b$. The tangent bundles $TZ_b$ of fibers form a subbundle of $TN$, known as the vertical tangent bundle and denoted by $T^VN$.
	
Let $p\colon\mE\to N$ be a complex vector bundle. Then the composition $\pi\circ p\colon\mE\to B$ defines a fiber bundle, whose fiber over $b\in B$ is denoted by $E_b$. The restriction $p|_{E_b}\colon E_b\to Z_b$ defines a complex vector bundle over $Z_b$ for each $b\in B$. It can be shown that all these bundles are smoothly isomorphic to a fixed complex vector bundle $E\to Z$. These settings coincide with those considered in \cite[Definition 1.2]{ASIV}.

Assume that the fiber $Z$ is a closed manifold and let $\big(\mE,h^{\mE}\big)$ be an Hermitian vector bundle over $N$. Given a Euclidean metric $g^{T^VN}$ on the vertical tangent bundle $T^VN$, this induces a Riemannian metric $g_b$ on each fiber $Z_b$. With respect to this metric, we define the Hilbert space $L^2(Z_b,E_b)$ of square-integrable sections of $E_b\to Z_b$ for each $b\in B$. Since $Z$ is compact, all the spaces $L^2(Z_b,E_b)$ are mutually isomorphic and together form a Hilbert bundle over $B$, which we denote by $\mL^2(N,\mE)$. Similarly, for another Hermitian vector bundle $\left(\mF,h^{\mF}\right)\to N$, we obtain the Hilbert bundle $\mL^2(N,\mF)\to B$.

Denote by $F_b$ the fiber of $\mF$ over $b\in B$. Let $P\coloneq\left\{P_b\right\}_{b\in B}$ be a smooth family of elliptic differential operators from $\mE$ to $\mF$, where $P_b$ maps $C^{\infty}(Z_b,E_b)$ to $C^{\infty}(Z_b,F_b)$. Then the collection
\begin{equation}\label{FamilyOp}
	\left\{\frac{P_b}{\sqrt{\Id+P_b^*P_b}}\colon L^2(Z_b,E_b)\to L^2(Z_b,F_b)\right\}_{b\in B}
\end{equation}
forms a smooth family of Fredholm operators from $\mL^2(N,\mE)$ to $\mL^2(N,\mF)$, where $P_b^*$ denotes the adjoint of $P_b$. According to the result of \cite[Proposition 2.2]{ASIV}, this family defines an index bundle in $K^0(B)$, which we denote by $\Ind(P)$.

Now assume that $\mE=\mF$ and that each $P_b$ in the family $P$ is elliptic and formally self-adjoint. Then $P_b$ is essentially self-adjoint in $L^2(Z_b,E_b)$. Following the spirit of \cite[Theorem B]{AS1969}, we consider the family of Fredholm operators:
\begin{equation}\label{FamilyOpSelf-adjoint}
	\left\{\Id\cos\pi s+\frac{\sqrt{-1}P_b}{\sqrt{\Id+P_b^2}}\sin\pi s\colon L^2(Z_b,E_b)\to L^2(Z_b,E_b)\right\}_{s\in[0,1],b\in B}.
\end{equation}
For each $b\in B$, the path
$$\left\{\Id\cos\pi s+\frac{\sqrt{-1}P_b}{\sqrt{\Id+P_b^2}}\sin\pi s\right\}_{s\in[0,1]}$$
starts at $\Id$ and ends at $-\Id$. Hence, the family in \eqref{FamilyOpSelf-adjoint} can be viewed as a family of Fredholm operators on the suspension
\begin{equation}\label{Suspension}
    \Sigma B\coloneq B\times[0,1]/\left(B\times\{0\}\cup B\times\{1\}\right)
\end{equation}
of $B$, and thus defines an index bundle in $K^0(\Sigma B)=K^1(B)$. We denote this class again by $\Ind(P)$.
	
Consider $N\times\R$, which is a fiber bundle over $B$ with fiber $Z_b\times\R$ over $b$. Denote the projection from $N\times\R$ to $N$ by $\pr$. Let $\big(\ov{\mE},h^{\ov{\mE}}\big)$ be an Hermitian vector bundle over $N\times\R$. We aim to define a Hilbert bundle associated with $\ov{\mE}\to N\times\R$, analogous to the construction of $\mL^2(N,\mE)$. The primary difficulty arises from the noncompactness of $Z\times\R$. To simplify the analysis, we impose additional conditions on $\ov{\mE}$ and the metrics, which are sufficient for the purposes of our current context.
	
\begin{defn}\label{AlmostProductStructure}
	Let $g^{T^V(N\times\R)}$ be a Euclidean metric on the vertical tangent bundle of $N\times\R$. We say that the Hermitian vector bundle $\big(\ov{\mE},h^{\ov{\mE}}\big)\to\big(N\times\R,g^{T^V(N\times\R)}\big)$ has an almost product structure if $\ov{\mE}$, $h^{\ov{\mE}}$ and $g^{T^V(N\times\R)}$ satisfy the following conditions:
	\begin{itemize}
		\item[$(a)$] There exists a Euclidean metric $g^{T^VN}$ on $T^VN$, such that $g^{T^V(N\times\R)}=\pr^*g^{T^VN}\oplus\dd t\otimes\dd t$ outside a compact subset of $N\times\R$. Here $t$ is the function defined by the projection $N\times\R\to\R$.
		\item[$(b)$] There exists an Hermitian vector bundle $\left(\mE,h^{\mE}\right)\to N$ such that $\ov{\mE}\cong \pr^*\mE$ as smooth bundles, and $h^{\ov{\mE}}=\pr^*h^{\mE}$ outside a compact subset of $N\times\R$.\footnote{Since $\R$ is contractible, there always exists a bundle $\mE$ satisfying the first condition in (b). Hence, this definition mainly concerns the Hermitian metric.}
	\end{itemize}
	Furthermore, if we wish to emphasize the structure, we refer to $\big(\ov{\mE},h^{\ov{\mE}}\big)\to\big(N\times\R,g^{T^V(N\times\R)}\big)$ as having an almost product structure of the type $\left(\mE,h^{\mE}\right)\to\big(N,g^{T^VN}\big)$.
\end{defn}
	
Under the above settings and denoting the fiber of the composite map $\ov{\mE}\to N\times\R\to B$ over $b\in B$ by $\ov{E}_b$, it is straightforward to verify that there is an isomorphism of Hilbert spaces
\begin{equation}\label{L2Iso}
	L^2(Z_b\times\R,\ov{E}_b)\cong L^2(Z_b,E_b)\otimes L^2(\R),
\end{equation}
where $L^2(\R)$ denotes the space of square-integrable complex-valued functions on $\R$. Although isomorphism (\ref{L2Iso}) is not an isometry in general, the norms on both sides are equivalent. We already know that $\bigcup\limits_{b \in B}L^2(Z_b,E_b)$ forms a Hilbert bundle over $B$. This leads to the following proposition.
	
\begin{prop}\label{HilbertBundle}
	Let $\big(\ov{\mE},h^{\ov{\mE}}\big)\to\big(N\times\R,g^{T^V(N\times\R)}\big)$ be a bundle having an almost product structure of the type $\big(\mE,h^{\mE}\big)\to\big(N,g^{T^VN}\big)$. Then, the collection of all $L^2$-spaces $L^2(Z_b\times\R,\ov{E}_b)$ forms a Hilbert bundle over $B$, which we denote by $\mL^2\big(N\times\R,\ov{\mE}\big)$. Moreover, there is a bundle isomorphism
	\begin{equation}
		\mL^2\left(N\times\R,\ov{\mE}\right)\cong \mL^2(N,\mE)\otimes L^2(\R),
	\end{equation}
	where $L^2(\R)$ is the trivial bundle $B\times L^2(\R)$ over $B$.
\end{prop}

Let $\big(\ov{\mF},h^{\ov{\mF}}\big)\to\big(N\times\R,g^{T^V(N\times\R)}\big)$ be another Hermitian vector bundle with an almost product structure, and denote by $\ov{F}_b$ the fiber of $\ov{\mF}\to N\times\R\to B$ over $b\in B$. Consider a smooth family of first-order elliptic differential operators
$$\ov{P}\coloneq\left\{\ov{P}_b\colon C^{\infty}\left(Z_b\times\R,\ov{E}_b\right)\to C^{\infty}\left(Z_b\times\R,\ov{F}_b\right)\right\}_{b\in B}.$$
Suppose that each $\ov{P}_b$ and its adjoint $\ov{P}_b^*$ are uniformly positive at infinity, i.e., there exist a compact subset $K_b\subset Z_b\times\R$ and a constant $c_b>0$ such that
\begin{equation}
    \left\|\ov{P}_b u\right\|\geqslant c_b\|u\|,\quad \left\|\ov{P}_b^*v\right\|\geqslant c_b\|v\|
\end{equation}
hold for all $u\in C_0^{\infty}\big(Z_b\times\R,\ov{E}_b\big)$ and $v\in C_0^{\infty}\big(Z_b\times\R,\ov{F}_b\big)$ whose supports are contained in $(Z_b\times\R)\setminus K_b$. Then a proof similar to that of \cite[Theorem 2.1]{Ang1993} or \cite[Theorem 3.2]{GL1983} shows that
\begin{equation}\label{FamilyOp2}
	\frac{\ov{P}_b}{\sqrt{\Id+\ov{P}_b^*\ov{P}_b}}\colon L^2\left(Z_b\times\R,\ov{E}_b\right)\to L^2\left(Z_b\times\R,\ov{F}_b\right)
\end{equation}
is a Fredholm operator for each $b\in B$. If, in addition, the Fredholm operators in \eqref{FamilyOp2} depend continuously on $b\in B$, then, following the arguments in \cite{ASIV}, we obtain an index bundle $\Ind\left(\ov{P}\right)\in K^0(B)$.

Now assume that $\ov{\mE}=\ov{\mF}$ and that $\ov{P}$ is a smooth family of formally self-adjoint first-order elliptic differential operators. If each $\ov{P}_b$ is uniformly positive at infinity and the Fredholm operators
\begin{equation}\label{FamilyOpSelfadjoint2}
	\frac{\ov{P}_b}{\sqrt{\Id+\ov{P}_b^2}}\colon L^2\left(Z_b\times\R,\ov{E}_b\right)\to L^2\left(Z_b\times\R,\ov{F}_b\right)
\end{equation}
depend continuously on $b\in B$, then, as in the compact case discussed previously, the family in \eqref{FamilyOpSelf-adjoint}, with $P_b$ replaced by $\ov{P}_b$ and $E_b$ replaced by $\ov{E}_b$, determines an index bundle $\Ind\big(\ov{P}\big)\in K^1(B)$ by \cite{AS1969}.

\subsection{Twisted Spinor Bundles and Families of Twisted Dirac Operators}\label{Subsection1.2}
In this subsection, we consider a fiber bundle $\pi\colon N\to B$ with connected, closed fiber $Z$, and denote the fiber over $b\in B$ by $Z_b$. We assume that the vertical tangent bundle $T^V N$ is oriented and spin, which implies that the fiber $Z$ is spin. Our focus will be on the twisted spinor bundles and the families of twisted Dirac operators associated with the vertical tangent bundles of $N\to B$ and $N\times\R\to B$.
	
We fix a Euclidean metric $g$ on $T^VN$, and denote by $\{g_b\}_{b\in B}$ its restrictions to the fibers. This enables us to define the complex spinor bundle $S\big(T^VN\big)\to N$ associated with $g$,\footnote{Throughout this paper, all spinor bundles are understood to be complex.} which is an Hermitian vector bundle whose metric is induced by $g$. For simplicity, we write
$$\mS\coloneq S\big(T^VN\big).$$
Denote the fiber of $\mS\to B$ over $b\in B$ by $S_b$. Then $S_b$ is the spinor bundle over $Z_b$ associated with the metric $g_b$. Let $\Cl\left(T^V N\right)$ be the Clifford bundle of $T^V N$, defined analogously to \cite[Chapter II, Definition 3.4]{LM1989}. Let $c(e)\in\Cl\left(T^V N\right)$ represent the element corresponding to $e\in T^V N$. Our convention for the Clifford algebra is
\begin{equation}
	c(e_1)c(e_2)+c(e_2)c(e_1)=-2\langle e_1,e_2 \rangle_g.
\end{equation}
It is known that $\mS$ is a $\Cl\left(T^V N\right)$-module and $c(e)$ acts unitarily on $\mS$ for $|e|_g=1$.
	
There is a homomorphism
\begin{equation}\label{tau}
	\tau\coloneq\left(\sqrt{-1}\right)^{\left[\frac{m+1}{2}\right]} c(e_{b,1})\cdots c(e_{b,m}),
\end{equation}
defined for a local oriented orthonormal frame $\{e_{b,1},\dots,e_{b,m}\}$ of $TZ_b$, where $m\coloneq\dim Z$.\footnote{Here we use $[s]$ to denote the greatest integer less than or equal to $s$ for $s\in\R$.} It is straightforward to verify that $\tau$ defines a global section of the endomorphism bundle $\End(S_b)\to Z_b$.
	
In the case where $m$ is even, we compute that $\tau^2=\Id$. Thus, there is a $\Z_2$-grading
\begin{equation}\label{Z2GradingOriginal}
	S_b=S_{b,+}\oplus S_{b,-}
\end{equation}
on $S_b$, where $S_{b,\pm}$ are the eigenspaces corresponding to the eigenvalues $\pm1$ of $\tau$. Combining all of this, we obtain the decomposition $\mS=\mS_+\oplus\mS_-$ of the Hermitian vector bundle $\mS$.
	
For odd $m$, there are in fact two isomorphic choices for the spinor bundle $\mS$, corresponding to the two inequivalent irreducible representations of the Clifford algebra \cite[Section II.3]{LM1989}. These two bundles are characterized by the relation $\tau=\pm\Id$.
	
Let $\nabla^{\mS}$ be the Hermitian connection on $\mS$ induced by the fiberwise Levi-Civita connections $\nabla^{TZ_b}$ on $Z_b$. Denote the restriction of $\nabla^{\mS}$ to the fiber $S_b$ by $\nabla^{S_b}$ for $b\in B$. This connection satisfies the following compatibility condition
\begin{equation}\label{SpinConnection}
	\nabla^{S_b}_{e_1}\left(c(e_2)s\right)=c\left(\nabla^{TZ_b}_{e_1}e_2\right)s+c(e_2)\nabla^{S_b}_{e_1}s,
\end{equation}
where $s$ is a smooth section of $S_b$, and $e_1$, $e_2$ are smooth sections of $TZ_b$ \cite[Section II.4]{LM1989}.
	
Let $\mE\to N$ be an Hermitian vector bundle equipped with an Hermitian connection $\nabla^{\mE}$. We denote by $\nabla^{E_b}$ the restriction of $\nabla^{\mE}$ to the fiber $E_b$ of $\mE$ over $b\in B$. The tensor product $\mS\otimes\mE$ is again an Hermitian vector bundle over $N$, whose restriction to each fiber $Z_b$ is the twisted spinor bundle $S_b\otimes E_b$ for each $b\in B$. We define the corresponding family of twisted Dirac operators to be $D^{\mE}\coloneq\big\{D^{E_b}\big\}_{b\in B}$. More precisely, if $\{e_{b,1},\dots,e_{b,m}\}$ is a local orthonormal frame of $TZ_b$, the twisted Dirac operator on $E_b$ is defined by
\begin{equation}\label{TwistedDiracOp}
	D^{E_b}\coloneq\sum_{k=1}^{m}c(e_{b,k})\nabla^{S_b\otimes E_b}_{e_{b,k}}=\sum_{k=1}^{m}c(e_{b,k})\left(\nabla^{S_b}_{e_{b,k}}\otimes\Id_{E_b}+\Id_{S_b}\otimes\nabla^{E_b}_{e_{b,k}}\right).
\end{equation}
It is straightforward to verify that the formula (\ref{TwistedDiracOp}) is independent of the choices of the local orthonormal frames, so $D^{E_b}$ is well-defined.
	
It is known that $D^{E_b}$ is a formally self-adjoint first-order elliptic differential operator. When $m$ is even, the $\Z_2$-grading in (\ref{Z2GradingOriginal}) induces a $\Z_2$-grading on the twisted spinor bundles
\begin{equation}
	S_b\otimes E_b=S_{b,+}\otimes E_b\oplus S_{b,-}\otimes E_b.
\end{equation}
Note that for any $e\in TZ_b$, the Clifford multiplication $c(e)$ interchanges this grading, and so does the twisted Dirac operator $D^{E_b}$. Hence, $D^{E_b}$ admits a decomposition
\begin{equation}
	D^{E_b}=
	\begin{pmatrix}
		0 & D^{E_b}_- \\
		D^{E_b}_+ & 0
	\end{pmatrix},
\end{equation}
where
$$D^{E_b}_{\pm}\colon C^{\infty}\left(Z_b,S_{b,\pm}\otimes E_b\right)\to C^{\infty}\left(Z_b,S_{b,\mp}\otimes E_b\right),$$
and $D^{E_b}_+$ and $D^{E_b}_-$ are formally adjoint to each other. In the family case, we also write
\begin{equation}
	D^{\mE}=
	\begin{pmatrix}
		0 & D^{\mE}_- \\
		D^{\mE}_+ & 0
	\end{pmatrix}.
\end{equation}
	
Next, we consider the fiber bundle $N\times\R\to B$ with fiber $Z\times\R$ of dimension $m+1$. Denote the projections $N\times\R\to\R$  and $N\times\R\to N$ by $t$ and $\pr$, respectively. Endow the vertical tangent bundle $T^V(N\times\R)$ with the orientation such that  $\big\{e_1,\dots,e_m,\frac{\p}{\p t}\big\}$ is oriented whenever $\{e_1,\dots,e_m\}$ is a local oriented frame of $T^VN\subset T^V(N\times\R)$.
	
Note that $T^V(N\times\R)$ remains spin. Given a Euclidean metric $\bar{g}$ on $T^V(N\times\R)$, we denote the associated spinor bundle by $\ov{\mS}\coloneq S\big(T^V(N\times\R)\big)$. For each $b\in B$, the restriction of $\ov{\mS}$ to $Z_b\times\R$ is the spinor bundle $\ov{S}_b$ corresponding to the restriction metric $\bar{g}_b$ of $\bar{g}$ to $T(Z_b\times\R)$. We denote the Clifford action of $\Cl\left(T^V(N\times\R)\right)$ on $\ov{\mS}$ by $\bar{c}(\cdot)$. There is an analogue of the homomorphism $\tau$ defined in \eqref{tau}, which we denote by $\bar{\tau}$:
\begin{equation}\label{taubar}
	\bar{\tau}\coloneq\left(\sqrt{-1}\right)^{\left[\frac{m+2}{2}\right]}\,\bar{c}(e_{b,1})\cdots\bar{c}(e_{b,m+1})
\end{equation}
for a local oriented orthonormal frame $\{e_{b,1},\dots,e_{b,m+1}\}$ of $T(Z_b\times\R)$.
	
Similar to the compact case, there exists an Hermitian connection $\ov{\nabla}^{\ov{\mS}}$ on $\ov{\mS}$ induced by the fiberwise Levi-Civita connections. We denote by $\Big\{\ov{\nabla}^{\ov{S}_b}\Big\}_{b\in B}$ the restrictions of $\ov{\nabla}^{\ov{\mS}}$ to the fibers. 
	
Let $\ov{\mE}$ be an Hermitian vector bundle over $N\times\R$, equipped with an Hermitian connection. Denote by $\ov{E}_b$ the fiber of $\ov{\mE}$ over $b\in B$. In analogy with (\ref{TwistedDiracOp}), we define the family of twisted Dirac operators $\ov{D}^{\ov{\mE}}=\Big\{\ov{D}^{\ov{E}_b}\Big\}_{b\in B}$, where $\ov{D}^{\ov{E}_b}$ acts on $\ov{S}_b\otimes\ov{E}_b$. As discussed earlier, each $\ov{D}^{\ov{E}_b}$ is formally self-adjoint. Moreover, when $m$ is odd, there is a $\Z_2$-grading $\ov{\mS}\otimes\ov{\mE}=\ov{\mS}_+\otimes\ov{\mE}\oplus\ov{\mS}_-\otimes\ov{\mE}$ induced by $\bar{\tau}$ in (\ref{taubar}) and the family $\ov{D}^{\ov{\mE}}$ interchanges this grading.
	
Now consider the special case where $\bar{g}=\pr^*g\oplus\dd t\otimes\dd t$ and $\ov{\mE}=\pr^*\mE$ is equipped with the pullback metric and pullback connection. There is a canonical embedding
\begin{equation}\label{j}
	j\colon N\to N\times\R
\end{equation}
with image $N\times\{0\}$, which is covered by a morphism of Clifford bundles (cf. \cite[Chapter I, Theorem 3.7]{LM1989}):
\begin{equation}\label{CliffordBundleMap}
	\begin{aligned}
		\Cl\left(T^VN\right) & \to\left.\Cl\left(T^V(N\times\R)\right)\right|_{N\times\{0\}} \\
		c(e)\ \quad & \mapsto\quad\quad\bar{c}\left(\frac{\p}{\p t}\right)\bar{c}(e).
	\end{aligned}
\end{equation}
Through this map, the restriction $\ov{\mS}\big|_{N\times\{0\}}$ becomes a $\Cl\left(T^VN\right)$-module.
	
When $m$ is even, dimension counting shows that there exists a bundle isomorphism
\begin{equation}\label{BundleIsomorphism}
	\mS\cong\ov{\mS}|_{N\times\{0\}}
\end{equation}
covering $j$. Moreover, due to the product metric on $T^V(N\times\R)$, the Hermitian vector bundle $\ov{\mS}$ is isometric to the pullback of $\mS$ via the projection $\pr$ with respect to the corresponding pullback metric, i.e.,
\begin{equation}\label{BundleIsometry}
	\ov{\mS}\cong\pr^*\mS.
\end{equation}
Consequently, $\ov{\mS}$, and hence $\ov{\mS}\otimes\ov{\mE}$, admit almost product structures.
	
When $m$ is odd, let $\mS_1$ denote the spinor bundle $\mS$ of $T^VN$ with $\tau=\Id$ and $\mS_2$ the one with $\tau=-\Id$. For a local oriented orthonormal frame $\{e_{b,1},\dots,e_{b,m}\}$ of $TZ_b=T(Z_b\times\{0\})$, we compute
\begin{equation}
	\begin{aligned}
		\tau=& \left(\sqrt{-1}\right)^{\frac{m+1}{2}} \bar{c}\left(\frac{\p}{\p t}\right)\bar{c}(e_{b,1})\cdots \bar{c}\left(\frac{\p}{\p t}\right)\bar{c}(e_{b,m}) \\
		=& -\left(\sqrt{-1}\right)^{\left[\frac{m+2}{2}\right]} \bar{c}(e_{b,1})\cdots \bar{c}(e_{b,m})\bar{c}\left(\frac{\p}{\p t}\right) \\
		=& -\bar{\tau}
	\end{aligned}
\end{equation}
under the identification in (\ref{CliffordBundleMap}). Dimension counting then gives bundle isomorphisms
\begin{equation}\label{BundleIsomorphism'}
	\mS_1\cong\ov{\mS}_-|_{N\times\{0\}},\quad \mS_2\cong\ov{\mS}_+|_{N\times\{0\}}
\end{equation}
covering $j$.\footnote{See \cite[Chapter I, Proposition 5.10]{LM1989} for details.} Moreover, the product metric on $T^V(N\times\R)$ yields bundle isometries
\begin{equation}\label{BundleIsometry'}
	\ov{\mS}_+ \cong \pr^*\mS_2, \quad \ov{\mS}_- \cong \pr^*\mS_1,
\end{equation}
where the bundles on the right-hand side carry the pullback metric. Hence, in this case $\ov{\mS}$ and $\ov{\mS}\otimes\ov{\mE}$ also admit almost product structures.

\section{Family Product Formula}\label{Section2}
In this section, we establish the product formula for index bundles, treating separately the cases of even and odd $\dim Z$.

\subsection{The Case of Even-Dimensional $Z$}\label{Subsection2.1}
Let the fiber bundle $\pi\colon N\to B$ with fiber $Z$ be defined as in Subsection \ref{Subsection1.2}. Still denote by $g$ the Euclidean metric on $T^VN$. Throughout this subsection, we assume that $m=\dim Z$ is even.

Let $\mE\to N$ be an Hermitian vector bundle equipped with an Hermitian connection $\nabla^{\mE}$. This defines a family of twisted Dirac operators $D^{\mE}=\big\{D^{E_b}\big\}_{b\in B}$ as in \eqref{TwistedDiracOp}. We already know that $D^{\mE}$ splits as
\begin{equation}
	D^{\mE}=
	\begin{pmatrix}
		0 & D^{\mE}_- \\
		D^{\mE}_+ & 0
	\end{pmatrix}.
\end{equation}
As introduced in Subsection \ref{Subsection1.1}, the family $\big\{D^{E_b}_+\big\}_{b\in B}$ defines an index bundle $\Ind\big(D^{\mE}_+\big)\in K^0(B)$.

Then we consider the fiber bundle $N\times\R\to B$, and denote the projections $N\times\R\to\R$ by $t$ and $N\times\R\to N$ by $\pr$ as before. Equip $T^V(N\times\R)$ with the metric $\bar{g}\coloneq\pr^*g\oplus\dd t\otimes\dd t$ and the orientation specified in Subsection \ref{Subsection1.2}. We denote the associated spinor bundle by $\ov{\mS}$, which is chosen to satisfy $\bar{\tau}=-\Id$. Let $\ov{\mE}\coloneq\pr^*\mE$ be the pullback bundle equipped with the pullback metric and pullback connection. As discussed in the last subsection, the twisted spinor bundle $\ov{\mS}\otimes\ov{\mE}$ admits an almost product structure.

Denote by $\ov{D}^{\ov{\mE}}=\Big\{\ov{D}^{\ov{E}_b}\Big\}_{b\in B}$ the family of twisted Dirac operators acting on $\ov{\mS}\otimes\ov{\mE}$. Unfortunately, this family is generally not uniformly positive at infinity. We need to deform it by a zeroth-order potential term.

Consider the family of first-order elliptic differential operators
\begin{equation}\label{A_bEven}
	\Big\{A_b\coloneq\ov{D}^{\ov{E}_b}+\sqrt{-1}t\Id\Big\}_{b\in B},
\end{equation}
where $t$ is the function on $Z_b\times\R$ induced by the projection $N\times\R\to\R$. Such operators are also referred to as Callias-type operators, following \cite{Ca1978}.

Observe that $A_b^*=\ov{D}^{\ov{E}_b}-\sqrt{-1}t\Id$. Using \eqref{SpinConnection} and \eqref{TwistedDiracOp}, we compute
\begin{equation}\label{A_b^*A_b}
	\left(\ov{D}^{\ov{E}_b}\mp \sqrt{-1}t\Id\right)\left(\ov{D}^{\ov{E}_b}\pm \sqrt{-1}t\Id\right)=\left(\ov{D}^{\ov{E}_b}\right)^2\pm \sqrt{-1}\bar{c}\left(\frac{\p}{\p t}\right)+t^2\Id.
\end{equation}
It follows that $A_b$ and $A_b^*$ are uniformly positive at infinity. 

Fix any $b_0\in B$ and choose a sufficiently small neighborhood $U_{b_0}$ of $b_0$ over which the fiber bundle $N\times\R\to B$ and the vector bundle $\ov{\mE}$ are trivialized. Let $\{b_n\}_{n\in\N^*}\subset U_{b_0}$ be a sequence converging to $b_0$. Via this trivialization, we regard the operators $A_{b_n}$ and $\ov{D}^{\ov{E}_{b_n}}$ as acting on sections over $Z_{b_0}\times\R$.

Since $T^N(N\times\R)$ is equipped with the product metric and $\ov{\mE}$ with the pullback metric, there exists a sequence of positive numbers $\{c_n\}_{n\in\N^*}$ with $c_n\to 0$ such that for any $u_{b_0}\in C_0^{\infty}\big(Z_{b_0}\times\R,\ov{S}_{b_0}\otimes\ov{E}_{b_0}\big)$,
\begin{equation}\label{GraphEstimate1}
	\left\|\left(A_{b_n}-A_{b_0}\right)u_{b_0}\right\|
	=
	\left\|\left(\ov{D}^{\ov{E}_{b_n}}
	-\ov{D}^{\ov{E}_{b_0}}\right)u_{b_0}\right\|
	\leq c_n\left\|u_{b_0}\right\|_{H^1},
\end{equation}
where $\|\cdot\|_{H^1}$ denotes the $H^1$-norm. Equation \eqref{A_b^*A_b} gives
\begin{equation}\label{GraphEstimate2}
	\left\|A_{b_0}u_{b_0}\right\|^2\geqslant\left\|\ov{D}^{\ov{E}_{b_0}}u_{b_0}\right\|^2-\left\|u_{b_0}\right\|^2.
\end{equation}
Combining this with \eqref{GraphEstimate1} and the elliptic estimate for $\ov{D}^{\ov{E}_{b_0}}$, we obtain a constant $c>0$ such that
\begin{equation}\label{GraphEstimateFinal}
	\left\|\left(A_{b_n}-A_{b_0}\right)u_{b_0}\right\|\leqslant c_n\cdot c\left(\left\|\ov{D}^{\ov{E}_{b_0}}u_{b_0}\right\|+\left\|u_{b_0}\right\|\right)\leqslant 2c\cdot c_n\left(\left\|u_{b_0}\right\|+\left\|A_{b_0}u_{b_0}\right\|\right).
\end{equation}

The same argument applies to the adjoint family $\{A_b^*\}_{b\in B}$. Therefore, by \cite[Lemma 3.8]{Ebert2017} (or \cite[Proposition 1.7]{Nicolaescu2007}) and by following the proof of \cite[Proposition 3.7]{Ebert2017}, we conclude that the Fredholm operator
$$\frac{A_b}{\sqrt{\Id+A_b^*A_b}}\colon L^2\left(Z_b\times\R,\ov{S}_b\otimes\ov{E}_b\right)\to L^2\left(Z_b\times\R,\ov{S}_b\otimes\ov{E}_b\right)$$
depends continuously on $b\in B$.

Applying Proposition \ref{HilbertBundle} to $\ov{\mS}\otimes\ov{\mE}$ and following the argument given at the end of Subsection \ref{Subsection1.1}, we conclude that the family in \eqref{A_bEven} defines an index bundle $\Ind(A)\in K^0(B)$.

We now present a proposition from which the main theorem of this subsection essentially follows.

\begin{prop}\label{FamilyProductFormula1}
	Let $\Ind(A)$ and $\Ind\big(D^{\mE}_+\big)$ be defined as above. Then we have
	\begin{equation}
		\Ind(A)=\Ind\left(D^{\mE}_+\right)\in K^0(B).
	\end{equation}
\end{prop}

To prove this proposition, we first need a lemma.

Let $\mH_1,\mH_1',\mH_2,\mH_2'$ be Hilbert bundles over a connected closed manifold $B$, with fibers $H_{1,b},H_{1,b}',H_{2,b},H_{2,b}'$ over $b\in B$, respectively. Consider two continuous families of Fredholm operators
$$P=\left\{P_b\colon H_{1,b}\to H_{1,b}'\right\}_{b\in B},\quad Q=\left\{Q_b\colon H_{2,b}\to H_{2,b}'\right\}_{b\in B}.$$
Suppose there exist two continuous families of bounded linear operators,
$$\varphi=\left\{\varphi_b\colon H_{1,b}\to H_{2,b}\right\}_{b\in B},\quad \psi=\left\{\psi_b\colon H_{1,b}'\to H_{2,b}'\right\}_{b\in B},$$
making the diagram in Figure \ref{IndBunLemFig1} commute for all $b\in B$. Then, via diagram chasing (cf. \cite[Proposition 2.70 and 2.71]{Rotman}), we obtain induced maps $\mu_b\colon\Ker P_b\to\Ker Q_b$ and $\eta_b\colon\Coker P_b\to\Coker Q_b$ from $\varphi_b$ and $\psi_b$, respectively.

\begin{figure}[hbtp]
	\centering
	\begin{tikzcd}[row sep=large, column sep=large]
		H_{1,b} \arrow[d, "\varphi_b"'] \arrow[r, "P_b"]
		&H_{1,b}' \arrow[d, "\psi_b"]
		\\
		H_{2,b} \arrow[r, "Q_b"]
		&H_{2,b}'
	\end{tikzcd}
	\caption{}\label{IndBunLemFig1}
\end{figure}

\begin{lemma}\label{IndBunLem}
	If the maps $\mu_b$ and $\eta_b$ are isomorphisms for all $b\in B$, then the index bundles determined by $P$ and $Q$ coincide, i.e.,
	\begin{equation}
		\Ind(P)=\Ind(Q)\in K^0(B).
	\end{equation}
\end{lemma}

The proof of this lemma is somewhat involved and may interrupt the main flow of the article. For this reason, we have included the complete proof in Appendix.

Now we apply Lemma \ref{IndBunLem} to prove Proposition \ref{FamilyProductFormula1}. The calculations on the fibers in the proof are partially inspired by those in \cite{Ang19932}.

\begin{proof}[Proof of Proposition \ref{FamilyProductFormula1}]
	For convenience, we define
	$$P_{b,\pm}\coloneq\frac{D^{E_b}_\pm}{\sqrt{\Id+D^{E_b}_\mp D^{E_b}_\pm}},\quad Q_b\coloneq\frac{A_b}{\sqrt{\Id+A_b^*A_b}}.$$
	Due to Lemma \ref{IndBunLem}, it is sufficient to find two smooth families of bounded linear maps,
	$$\varphi=\left\{\varphi_b\colon L^2\big(Z_b,S_{b,+}\otimes E_b\big)\to L^2\big(Z_b\times\R,\ov{S}_b\otimes\ov{E}_b\big)\right\}_{b\in B}$$
	and
	$$\psi=\left\{\psi_b\colon L^2\big(Z_b,S_{b,-}\otimes E_b\big)\to L^2\big(Z_b\times\R,\ov{S}_b\otimes\ov{E}_b\big)\right\}_{b\in B},$$
	such that:
	\begin{itemize}
		\item[$(a)$] The diagram, as shown in Figure \ref{PropDiagram1}, commutes for all $b\in B$;
		\item[$(b)$] $\varphi_b$ induces an isomorphism between the kernels of $P_{b,+}$ and $Q_b$, while $\psi_b$ induces an isomorphism between the cokernels of $P_{b,+}$ and $Q_b$.
	\end{itemize}
	
	\begin{figure}[hbtp]
		\centering
		\begin{tikzcd}[row sep=large]
			L^2\big(Z_b,S_{b,+}\otimes E_b\big) \arrow[r, "\varphi_b"] \arrow[d, "P_{b,+}"']
			&L^2\left(Z_b\times\R,\ov{S}_b\otimes \ov{E}_b\right) \arrow[d, "Q_b"] \\
			L^2\big(Z_b,S_{b,-}\otimes E_b\big) \arrow[r, "\psi_b"]
			&L^2\left(Z_b\times\R,\ov{S}_b\otimes \ov{E}_b\right)
		\end{tikzcd}
		\caption{}\label{PropDiagram1}
	\end{figure}
	
	It is more convenient to consider under the isomorphism (\ref{BundleIsomorphism}), where we identify $S_b$ with the restriction of $\ov{S}_b$ to $Z_b\times\{0\}$. Due to the product metric on $T^V(N\times\R)$, we additionally have
	\begin{equation}
		\ov{\nabla}^{\ov{S}_b}=\left(\pr_b\right)^*\nabla^{S_b}.
	\end{equation}
	Thus, under the above identification and (\ref{CliffordBundleMap}), the twisted Dirac operator in (\ref{TwistedDiracOp}) can be reformulated as
	\begin{equation}\label{DiracReformulate}
		D^{E_b}=\sum_{k=1}^{m} \bar{c}\left(\frac{\p}{\p t}\right)\bar{c}(e_{b,k})\ov{\nabla}^{\ov{S}_b\otimes \ov{E}_b}_{e_{b,k}}
	\end{equation}
	acting on
	$$C^{\infty}\left(Z_b\times\{0\},\ov{S}_b\otimes\ov{E}_b\right)\coloneq C^{\infty}\left(Z_b\times\{0\},\left.\left(\ov{S}_b\otimes\ov{E}_b\right)\right|_{Z_b\times\{0\}}\right),$$
	for a local orthonormal basis $\{e_{b,1},\dots,e_{b,m}\}$ of $T(Z_b\times\{0\})$. Moreover, the endomorphism $\tau$ defined in (\ref{tau}) becomes
	\begin{equation}\label{TauReformulate}
		\begin{aligned}
			\tau& =\left(\sqrt{-1}\right)^{\frac{m}{2}}\bar{c}\left(\frac{\p}{\p t}\right)\bar{c}(e_{b,1})\cdots\bar{c}\left(\frac{\p}{\p t}\right)\bar{c}(e_{b,m}) \\
			&=\left(\sqrt{-1}\right)^{\frac{m}{2}}\bar{c}(e_{b,1})\cdots\bar{c}(e_{b,m}) \\
			&=\sqrt{-1}\bar{\tau}\bar{c}\left(\frac{\p}{\p t}\right)
		\end{aligned}
	\end{equation}
	under (\ref{CliffordBundleMap}), where $\{e_{b,1},\dots,e_{b,m}\}$ is a local oriented orthonormal basis of $T(Z_b\times\{0\})$. Due to the product structure, this homomorphism extends to the entire spinor bundle $\ov{S}_b$, and we have a $\Z_2$-grading
	\begin{equation}\label{Z2GradingTwist}
		\ov{S}_b\otimes\ov{E}_b=\ov{S}_{b,+}\otimes\ov{E}_b\oplus\ov{S}_{b,-}\otimes\ov{E}_b.
	\end{equation}
	
	Under these settings, we have
	\begin{align*}
		P_{b,\pm}&\colon L^2\left(Z_b\times\{0\},\ov{S}_{b,\pm}\otimes \ov{E}_b\right)\to L^2\left(Z_b\times\{0\},\ov{S}_{b,\mp}\otimes \ov{E}_b\right), \\
		\varphi_b&\colon L^2\left(Z_b\times\{0\},\ov{S}_{b,+}\otimes\ov{E}_b\right)\to L^2\left(Z_b\times\R,\ov{S}_b\otimes\ov{E}_b\right), \\
		\psi_b&\colon L^2\left(Z_b\times\{0\},\ov{S}_{b,-}\otimes\ov{E}_b\right)\to L^2\left(Z_b\times\R,\ov{S}_b\otimes\ov{E}_b\right),
	\end{align*}
	and Figure \ref{PropDiagram1} changes into Figure \ref{PropDiagram1'} as follows.
	
	\begin{figure}[hbtp]
		\centering
		\begin{tikzcd}[row sep=large]
			L^2\left(Z_b\times\{0\},\ov{S}_{b,+}\otimes \ov{E}_b\right) \arrow[r, "\varphi_b"] \arrow[d, "P_{b,+}"']
			&L^2\left(Z_b\times\R,\ov{S}_b\otimes \ov{E}_b\right) \arrow[d, "Q_b"] \\
			L^2\left(Z_b\times\{0\},\ov{S}_{b,-}\otimes \ov{E}_b\right) \arrow[r, "\psi_b"]
			&L^2\left(Z_b\times\R,\ov{S}_b\otimes \ov{E}_b\right)
		\end{tikzcd}
		\caption{}\label{PropDiagram1'}
	\end{figure}
	
	To construct $\varphi_b$ and $\psi_b$, we first introduce some notation.
	
	Denote the Schwartz space of complex-valued functions on $\R$ by $\cS(\R)$. Let
	$$C^{\infty}_{rd}\left(Z_b\times\R,\ov{S}_b\otimes\ov{E}_b\right)\coloneq C^{\infty}\left(Z_b\times\{0\},\ov{S}_b\otimes\ov{E}_b\right)\otimes\cS(\R)$$
	be the set of rapidly decaying sections of $\ov{S}_b\otimes\ov{E}_b\to Z_b\times\R$. We already know the isomorphism
	\begin{equation}\label{L2Iso2}
		L^2\left(Z_b\times\R,\ov{S}_b\otimes\ov{E}_b\right)\cong L^2\left(Z_b\times\{0\},\ov{S}_b\otimes\ov{E}_b\right)\otimes L^2(\R)
	\end{equation}
	from (\ref{L2Iso}). Moreover, from (\ref{BundleIsometry}) and the assumption that $\ov{\mE}$ is a pullback bundle equipped with the pullback metric, (\ref{L2Iso2}) is in fact an isometry. This implies that $C^{\infty}_{rd}\big(Z_b\times\R,\ov{S}_b\otimes\ov{E}_b\big)$ is a dense subspace of  $L^2\big(Z_b\times\R,\ov{S}_b\otimes\ov{E}_b\big)$.	Similar statements hold when $\ov{S}_b$ is replaced by $\ov{S}_{b,\pm}$.	
	
	Define the map $\varphi_b$ as\footnote{More precisely, the image should be written as $e^{-\frac{1}{2}t^2}\ov{\pr}_b^*u_b$, where $\ov{\pr}_b\colon Z_b\times\R\to Z_b\times\{0\}$ is the projection. For simplicity, we omit this additional notation.}
	\begin{equation}
		\begin{aligned}
			\varphi_b\colon C^{\infty}\left(Z_b\times\{0\},\ov{S}_{b,+}\otimes \ov{E}_b\right)&\to C^{\infty}_{rd}\left(Z_b\times\R,\ov{S}_b\otimes \ov{E}_b\right) \\
			u_b\qquad\qquad\quad&\mapsto\qquad\quad e^{-\frac{1}{2}t^2}u_b.
		\end{aligned}
	\end{equation}
	It is straightforward to verify that $\varphi_b$ is well-defined, and its completion gives a bounded map from $L^2\left(Z_b \times \{0\},\ov{S}_{b,+}\otimes\ov{E}_b\right)$ to $L^2\left(Z_b\times\R,\ov{S}_b\otimes\ov{E}_b\right)$, which we continue to denote by $\varphi_b$. Similarly, we define
	\begin{equation}
		\begin{aligned}
			\psi_b\colon L^2\left(Z_b\times\{0\},\ov{S}_{b,-}\otimes \ov{E}_b\right)&\to L^2\left(Z_b\times\R,\ov{S}_b\otimes \ov{E}_b\right) \\
			v_b\qquad\qquad\quad&\mapsto\qquad \sqrt{-1}e^{-\frac{1}{2}t^2}v_b.
		\end{aligned}
	\end{equation}
	
	Next, we verify that these two maps make the diagram in Figure \ref{PropDiagram2} commute.
	
	Note that from (\ref{TauReformulate}) we see that the action $\bar{c}(e)$ interchanges the grading in (\ref{Z2GradingTwist}) for $e\in T(Z_b\times\{0\})\subset T(Z_b\times\R)$, while equation $\bar{\tau}=-\Id$ implies that $\bar{c}\left(\frac{\p}{\p t}\right)$ corresponds to scalar multiplication by $\pm\sqrt{-1}$ on $\ov{S}_{b,\pm}\otimes\ov{E}_b$, respectively. Clearly, the image of $\varphi_b$ is contained in $L^2\big(Z_b\times\R,\ov{S}_{b,+}\otimes \ov{E}_b\big)$ while the image of $\psi_b$ is contained in $L^2\big(Z_b\times\R,\ov{S}_{b,-}\otimes \ov{E}_b\big)$.
	
	Take an arbitrary $u_b$ in $C^{\infty}\left(Z_b\times\{0\},\ov{S}_{b,+}\otimes\ov{E}_b\right)$. For the upper square, when traversing clockwise, we compute using a local orthonormal basis $\{e_{b,1},\dots,e_{b,m}\}$ of $T(Z_b\times\{0\})$ that
    \begin{equation}
        \begin{aligned}
		    A_b\circ\varphi_b(u_b)&=\left(\ov{D}^{\ov{E}_b}+\sqrt{-1}t\Id\right)\left(e^{-\frac{1}{2}t^2}u_b\right) \\
		    &=\sum_{k=1}^{m}\bar{c}(e_{b,k})\ov{\nabla}^{\ov{S}_b\otimes \ov{E}_b}_{e_{b,k}}\left(e^{-\frac{1}{2}t^2}u_b\right) +\bar{c}\left(\frac{\p}{\p t}\right)\frac{\p}{\p t}\left(e^{-\frac{1}{2}t^2}u_b\right) +\sqrt{-1}t\cdot e^{-\frac{1}{2}t^2}u_b \\
		    &=e^{-\frac{1}{2}t^2}\cdot\sum_{k=1}^{m}\bar{c}(e_{b,k})\ov{\nabla}^{\ov{S}_b\otimes \ov{E}_b}_{e_{b,k}}u_b.
	    \end{aligned}
    \end{equation}
	When traversing counterclockwise, we have
    \begin{equation}
        \begin{aligned}
		    \psi_b\circ D^{E_b}_+(u_b) &=\sqrt{-1}e^{-\frac{1}{2}t^2}\cdot\sum_{k=1}^{m}\bar{c}\left(\frac{\p}{\p t}\right)\bar{c}(e_{b,k})\ov{\nabla}^{\ov{S}_b\otimes \ov{E}_b}_{e_{b,k}}u_b \\
		    &=e^{-\frac{1}{2}t^2}\cdot\sum_{k=1}^{m}\bar{c}(e_{b,k})\ov{\nabla}^{\ov{S}_b\otimes \ov{E}_b}_{e_{b,k}}u_b.
	    \end{aligned}
    \end{equation}
	
	Thus, the upper square in Figure \ref{PropDiagram2} commutes. The verification for the lower square follows in a similar manner.
	
	\begin{figure}[hbtp]
		\centering
		\begin{tikzcd}[row sep=large]
			C^{\infty}\left(Z_b\times\{0\},\ov{S}_{b,+}\otimes \ov{E}_b\right) \arrow[r, "\varphi_b"] \arrow[d, "D^{E_b}_+"']
			&C^{\infty}_{rd}\left(Z_b\times\R,\ov{S}_b\otimes \ov{E}_b\right) \arrow[d, "A_b"] \\
			C^{\infty}\left(Z_b\times\{0\},\ov{S}_{b,-}\otimes \ov{E}_b\right) \arrow[r, "\psi_b"] \arrow[d, "D^{E_b}_-"']
			&C^{\infty}_{rd}\left(Z_b\times\R,\ov{S}_b\otimes \ov{E}_b\right) \arrow[d, "A_b^*"] \\
			C^{\infty}\left(Z_b\times\{0\},\ov{S}_{b,+}\otimes \ov{E}_b\right) \arrow[r, "\varphi_b"]
			&C^{\infty}_{rd}\left(Z_b\times\R,\ov{S}_b\otimes \ov{E}_b\right)			
		\end{tikzcd}
		\caption{}\label{PropDiagram2}
	\end{figure}
	
	The commutativity of the diagram in Figure \ref{PropDiagram2} directly implies that the square in Figure \ref{PropDiagram3} also commutes. Combining this with the upper square in Figure \ref{PropDiagram2}, we conclude that the maps $\varphi_b$ and $\psi_b$ make the diagram in Figure \ref{PropDiagram1} commute. This proves $\varphi$ and $\psi$ satisfy $(a)$.
	
	\begin{figure}[hbtp]
		\centering
		\begin{tikzcd}[row sep=large]
			C^{\infty}\left(Z_b\times\{0\},\ov{S}_{b,+}\otimes \ov{E}_b\right) \arrow[r, "\varphi_b"] \arrow[d, "\Id+D^{E_b}_-D^{E_b}_+"']
			&C^{\infty}_{rd}\left(Z_b\times\R,\ov{S}_b\otimes \ov{E}_b\right) \arrow[d, "\Id+A_b^*A_b"] \\
			C^{\infty}\left(Z_b\times\{0\},\ov{S}_{b,+}\otimes \ov{E}_b\right) \arrow[r, "\varphi_b"]
			&C^{\infty}_{rd}\left(Z_b\times\R,\ov{S}_b\otimes \ov{E}_b\right)			
		\end{tikzcd}
		\caption{}\label{PropDiagram3}
	\end{figure}
	
	Let $\bar{\varphi}_b\colon\Ker P_{b,+}\to\Ker Q_b$ be the map induced by $\varphi_b$ via diagram chasing. It can be regarded as the restriction of $\varphi_b$ to $\Ker P_{b,+}$. Note that $\Ker P_{b,+}=\Ker D^{E_b}_+$ and $\Ker Q_b=\Ker A_b$, so we may simply write $\bar{\varphi}_b\colon\Ker D^{E_b}_+\to\Ker A_b$.	The injectivity of $\bar{\varphi}_b$ follows from the injectivity of $\varphi_b$. It remains to show that $\bar{\varphi}_b$ is surjective.
	
	Note that $A_b^*A_b$ preserves the $\Z_2$-grading in (\ref{Z2GradingTwist}) and
    \begin{equation}
        \begin{aligned}
		    A_b^*A_b&=\left(\ov{D}^{\ov{E}_b}\right)^2+ \sqrt{-1}\bar{c}\left(\frac{\p}{\p t}\right)+t^2\Id \\
		    &=\left(\sum_{k=1}^{m}\bar{c}(e_{b,k})\ov{\nabla}^{\ov{S}_b\otimes\ov{E}_b}_{e_{b,k}}\right)^2-\frac{\p^2}{\p t^2}+ \sqrt{-1}\bar{c}\left(\frac{\p}{\p t}\right)+t^2\Id \\
		    &=\left\{
		    \begin{aligned}
		    	&\left(\sum_{k=1}^{m}\bar{c}(e_{b,k})\ov{\nabla}^{\ov{S}_b\otimes\ov{E}_b}_{e_{b,k}}\right)^2 -\frac{\p^2}{\p t^2}+t^2\Id-\Id, & \text{on}\ C^{\infty}\left(Z_b\times\R,\ov{S}_{b,+}\otimes\ov{E}_b\right), \\
		    	&\left(\sum_{k=1}^{m}\bar{c}(e_{b,k})\ov{\nabla}^{\ov{S}_b\otimes\ov{E}_b}_{e_{b,k}}\right)^2-\frac{\p^2}{\p t^2}+t^2\Id+\Id, & \text{on}\ C^{\infty}\left(Z_b\times\R,\ov{S}_{b,-}\otimes\ov{E}_b\right), \\
		    \end{aligned}			
	    	\right.
	    \end{aligned}
    \end{equation}
	for a local orthonormal basis $\{e_{b,1},\dots,e_{b,m}\}$ of $T(Z_b\times\{0\})$. We know that the operator $\left(\sum\limits_{k=1}^{m}\bar{c}(e_{b,k})\ov{\nabla}^{\ov{S}_b\otimes\ov{E}_b}_{e_{b,k}} \right)^2$ is globally defined and nonnegative, and the harmonic oscillator $-\frac{\p^2}{\p t^2}+t^2\Id-\Id$ on $\R$ is nonnegative and has a one-dimensional kernel spanned by $e^{-\frac{1}{2}t^2}$, while $-\frac{\p^2}{\p t^2}+t^2\Id+\Id$ is strictly positive \cite[Theorem 1.5.1]{GJ1987}. Therefore, any element in $\Ker A_b=\Ker A_b^*A_b$ must have the form $e^{-\frac{1}{2}t^2}u_b$, where
	$$u_b\in C^{\infty}\left(Z_b\times\{0\},\ov{S}_{b,+}\otimes\ov{E}_b\right)\bigcap\Ker\left(\sum_{k=1}^{m}\bar{c}(e_{b,k})\ov{\nabla}^{\ov{S}_b\otimes\ov{E}_b}_{e_{b,k}}\right)=\Ker D^{E_b}_+.$$
	From this, we conclude that $\bar{\varphi}_b$ is surjective and thus an isomorphism.
	
	Now consider the map $\bar{\psi}_b\colon\Coker P_{b,+}\to\Coker Q_b$ induced by $\psi_b$. Since
	\begin{equation}
		\Coker P_{b,+}\cong\Ker P_{b,-}=\Ker D^{E_b}_-,\quad\Coker Q_b\cong\Ker Q_b^*=\Ker A_b^*
	\end{equation}
	and
	\begin{equation}
		\psi_b\left(\Ker D^{E_b}_-\right)\subseteq\Ker A_b^*,
	\end{equation}
	we may identify $\bar{\psi}_b$ with the restriction $\psi_b|_{\Ker D^{E_b}_-}\colon\Ker D^{E_b}_-\to\Ker A_b^*$. A similar argument to the one for $\bar{\varphi}_b$ then establishes that $\bar{\psi}_b$ is an isomorphism. It follows that $\varphi$ and $\psi$ satisfy $(b)$.
	
	By Lemma \ref{IndBunLem}, we complete the proof of this proposition. 
\end{proof}

In Proposition \ref{FamilyProductFormula1}, the Hermitian vector bundle $\ov{\mE}\to N\times\R$ that we consider has a product structure. In fact, we can extend this proposition to apply to a bundle with an almost product structure.

\begin{thm}[Family product formula for even-dimensional $Z$]\label{FamilyProductFormula2}
	Let the fiber bundles $(N,g)\to B$ and $\mS\to N$ be as defined at the beginning of this subsection. Suppose the fiber $Z$ of $N\to B$ is even-dimensional. Consider an Hermitian vector bundle $\left(\mE,h^{\mE}\right)$ over $N$ with an Hermitian connection $\nabla^{\mE}$, and denote the fiber of $\mE$ over $b\in B$ by $E_b$. This gives rise to the corresponding family of twisted Dirac operators $D^{\mE}_+=\big\{D^{E_b}_+\big\}_{b\in B}$ and an index bundle $\Ind\left(D^{\mE}_+\right)\in K^0(B)$.
	
	Let $\big(\ov{\mE},h^{\ov{\mE}}\big)\to(N\times\R,\bar{g})$ be a bundle with an almost product structure of the type $\big(\mE,h^{\mE}\big)\to(N,g)$, and let $\nabla^{\ov{\mE}}$ be an Hermitian connection on it that agrees with the pullback of $\nabla^{\mE}$ outside a compact subset of $N\times\R$. Denote the spinor bundle of $T^V(N\times\R)$ corresponding to $\bar{g}$, which satisfies $\bar{\tau}=-\Id$, by $\ov{\mS}\to N\times\R$. Then we obtain a family of twisted Dirac operators $\ov{D}^{\ov{\mE}}=\Big\{\ov{D}^{\ov{E}_b}\Big\}_{b\in B}$ associated with $\ov{\mE}$, where $\ov{E}_b$ is the fiber of $\ov{\mE}$ over $b\in B$. In this case, the bundle $\ov{\mS}\otimes\ov{\mE}$ has an almost product structure, and the family of Callias-type operators
	\begin{equation}\label{A_bEven2}
		A=\Big\{A_b\coloneq\ov{D}^{\ov{E}_b}+\sqrt{-1}t\Id\Big\}_{b\in B}
	\end{equation}
	induces an index bundle $\Ind(A)\in K^0(B)$.
	
	Then we have:
	\begin{equation}
		\Ind(A)=\Ind\left(D^{\mE}_+\right).
	\end{equation}
\end{thm}

\begin{proof}
	We can continuously deform the metric $\bar{g}$ to the product metric $\pr^*g\oplus\dd t\otimes\dd t$, and simultaneously deform the Hermitian metric $h^{\ov{\mE}}$ and the Hermitian connection $\nabla^{\ov{\mE}}$ to the pullback metric $\pr^*h^{\mE}$ and the pullback connection $\pr^*\nabla^{\mE}$, respectively, where $\pr\colon N\times\R\to N$ is the projection. This deformation only affects the data on a compact subset of $N\times\R$, so the bundle $\mL^2\big(N\times\R,\ov{\mS}\otimes\ov{\mE}\big)$ remains isomorphic, and $\Big\{\frac{A_b}{\sqrt{\Id+A_b^*A_b}}\Big\}_{b\in B}$ remains a continuous family of Fredholm operators. As a result, the index bundle $\Ind(A)\in K^0(B)$ is invariant under such a deformation. Therefore, this theorem follows from Proposition \ref{FamilyProductFormula1}.
\end{proof}

\subsection{The Case of Odd-Dimensional $Z$}\label{Subsection2.2}
Let $\pi\colon N\to B$ be a fiber bundle as defined in Subsection \ref{Subsection2.1}, but its fiber $Z$ is of odd dimension $m$. Equip $T^VN$ with the metric $g$. Recall from Subsection \ref{Subsection1.2} that there are two choices of the spinor bundle $\mS=S(T^VN)$ over $N$. We choose the one satisfying $\tau=-\Id$.

Consider an Hermitian vector bundle $\mE$ over $N$, equipped with an Hermitian connection $\nabla^{\mE}$. We still denote by $D^{\mE}\coloneq\big\{D^{E_b}\big\}_{b\in B}$ the family of twisted Dirac operators associated with the twisted spinor bundle $\mS\otimes\mE$. This family consists of formally self-adjoint elliptic differential operators, and therefore defines an index bundle $\Ind\big(D^{\mE}\big)\in K^1(B)$.

For the fiber bundle $N\times\R\to B$, we equip its vertical tangent bundle with the same product metric and orientation as before. Let $\ov{\mE}$ be the pullback bundle with the pullback metric and pullback connection. Then we obtain the corresponding twisted spinor bundle $\ov{\mS}\otimes\ov{\mE}$. Since $Z\times\R$ is even-dimensional, there is a $\Z_2$-grading
\begin{equation}
    \ov{\mS}\otimes\ov{\mE}=\ov{\mS}_+\otimes\ov{\mE}\oplus\ov{\mS}_-\otimes\ov{\mE}
\end{equation}
induced by $\bar{\tau}$ in (\ref{taubar}). As established in Subsection \ref{Subsection1.2}, $\ov{\mS}\otimes\ov{\mE}$ admits an almost product structure.

We continue to denote by $\ov{D}^{\ov{\mE}}=\Big\{\ov{D}^{\ov{E}_b}\Big\}_{b\in B}$ the family of twisted Dirac operators associated with $\ov{\mS}\otimes\ov{\mE}$. In the case of odd-dimensional $Z$, we use the following deformation of $\ov{D}^{\ov{\mE}}$:
\begin{equation}\label{A_bOdd}
	\left\{A_b\coloneq\ov{D}^{\ov{E}_b}+t\bar{\tau}\right\}_{b\in B},
\end{equation}
where $t$ is the function induced by the projection $N\times\R\to\R$ as previously defined.

Each operator in the family \eqref{A_bOdd} is formally self-adjoint. Since
\begin{equation}
	A_b^2=\left(\ov{D}^{\ov{E}_b}\right)^2+\bar{c}\left(\frac{\p}{\p t}\right)\bar{\tau}+t^2,
\end{equation}
we conclude that $A_b$ is positive at infinity. By an argument similar to that in Subsection \ref{Subsection2.1}, the family \eqref{A_bOdd} defines an index bundle $\Ind(A)\in K^1(B)$.

The analogous result to Proposition \ref{FamilyProductFormula1} for odd-dimensional $Z$ is stated as follows.

\begin{prop}\label{FamilyProductFormula1'}
	Let $\Ind(A)$ and $\Ind\big(D^{\mE}\big)$ be defined as above. Then
	\begin{equation}
		\Ind(A)=\Ind\left(D^{\mE}\right)\in K^1(B).
	\end{equation}
\end{prop}

Following the proof strategy of Proposition \ref{FamilyProductFormula1}, we first extend Lemma \ref{IndBunLem} to the $K^1$ case.

Let $\mH_1$ and $\mH_2$ be Hilbert bundles over a connected closed manifold $B$, with fibers $H_{1,b}$ and $H_{2,b}$ over $b\in B$, respectively. Consider continuous families of self-adjoint Fredholm operators
$$P=\left\{P_b\colon H_{1,b}\to H_{1,b}\right\}_{b\in B},\quad Q=\left\{Q_b\colon H_{2,b}\to H_{2,b}\right\}_{b\in B}.$$
Suppose there is a continuous family of bounded linear operators
$$\varphi=\left\{\varphi_b\colon H_{1,b}\to H_{2,b}\right\}_{b\in B}$$
making the diagram in Figure \ref{IndBunLem'Fig1} commute for all $b\in B$. Then, by diagram chasing, we obtain an induced map $\mu_b\colon\Ker P_b\to\Ker Q_b$.

\begin{figure}[hbtp]
	\centering
	\begin{tikzcd}[row sep=large, column sep=large]
		H_{1,b} \arrow[d, "\varphi_b"'] \arrow[r, "P_b"]
		&H_{1,b} \arrow[d, "\varphi_b"]
		\\
		H_{2,b} \arrow[r, "Q_b"]
		&H_{2,b}
	\end{tikzcd}
	\caption{}\label{IndBunLem'Fig1}
\end{figure}

\begin{lemma}\label{IndBunLem'}
	If the map $\mu_b$ is an isomorphism for all $b\in B$, then the index bundles determined by $P$ and $Q$ coincide, i.e.,
	\begin{equation}
		\Ind(P)=\Ind(Q)\in K^1(B).
	\end{equation}
\end{lemma}

The proof of this lemma is also given in Appendix.

We now proceed to prove Proposition \ref{FamilyProductFormula1'}.

\begin{proof}[Proof of Proposition \ref{FamilyProductFormula1'}]
	Write
	$$P_b=\frac{D^{E_b}}{\sqrt{\Id+\left(D^{E_b}\right)^2}},\quad Q_b=\frac{A_b}{\sqrt{\Id+A_b^2}}.$$
	Similar to (\ref{DiracReformulate}), we identify $D^{E_b}$ as
	\begin{equation}
		D^{E_b}=\sum_{k=1}^{m} \bar{c}\left(\frac{\p}{\p t}\right)\bar{c}(e_{b,k})\ov{\nabla}^{\ov{S}_b\otimes\ov{E}_b}_{e_{b,k}}
	\end{equation}
	acting on $C^{\infty}\left(Z_b\times\{0\},\ov{S}_{b,+}\otimes\ov{E}_b\right)$ for a local orthonormal basis $\{e_{b,1},\dots,e_{b,m}\}$ of $T(Z_b\times\{0\})$, via the map \eqref{CliffordBundleMap} and the isomorphism \eqref{BundleIsomorphism'}.
	
	Define the linear map
	\begin{equation}
		\begin{aligned}
			\varphi_b\colon L^2\left(Z_b\times\{0\},\ov{S}_{b,+}\otimes\ov{E}_b\right)&\to \quad L^2\left(Z_b\times\R,\ov{S}_b\otimes\ov{E}_b\right) \\
			u_b\quad\qquad\quad&\mapsto e^{-\frac{1}{2}t^2}u_b-\bar{c}\left(\frac{\p}{\p t}\right)e^{-\frac{1}{2}t^2}u_b.
		\end{aligned}
	\end{equation}
	We first verify that the diagram in Figure \ref{Prop'Diagram1} commutes.
	
	\begin{figure}[hbtp]
		\centering
		\begin{tikzcd}[row sep=large]
			C^{\infty}\left(Z_b\times\{0\},\ov{S}_{b,+}\otimes\ov{E}_b\right) \arrow[r, "\varphi_b"] \arrow[d, "D^{E_b}"']
			&C^{\infty}_{rd}\left(Z_b\times\R,\ov{S}_b\otimes\ov{E}_b\right) \arrow[d, "A_b"] \\
			C^{\infty}\left(Z_b\times\{0\},\ov{S}_{b,+}\otimes\ov{E}_b\right) \arrow[r, "\varphi_b"]
			&C^{\infty}_{rd}\left(Z_b\times\R,\ov{S}_b\otimes\ov{E}_b\right)
		\end{tikzcd}
		\caption{}\label{Prop'Diagram1}
	\end{figure}
	
	For $u_b\in C^{\infty}\left(Z_b\times\{0\},\ov{S}_{b,+}\otimes\ov{E}_b\right)$, when traversing clockwise, we have
    \begin{equation}\label{Clockwise'}
        \begin{aligned}
		    &A_b\circ\varphi_b(u_b) \\
		    =& \left(\sum_{k=1}^{m}\bar{c}(e_{b,k})\ov{\nabla}^{\ov{S}_b\otimes\ov{E}_b}_{e_{b,k}}+\bar{c}\left(\frac{\p}{\p t}\right)\frac{\p}{\p t}+\tau t\right)\left(e^{-\frac{1}{2}t^2}u_b-\bar{c}\left(\frac{\p}{\p t}\right)e^{-\frac{1}{2}t^2}u_b\right) \\
		    =& e^{-\frac{1}{2}t^2}\sum_{k=1}^{m}\bar{c}(e_{b,k})\ov{\nabla}^{\ov{S}_b\otimes\ov{E}_b}_{e_{b,k}}u_b-\bar{c}\left(\frac{\p}{\p t}\right)te^{-\frac{1}{2}t^2}u_b+te^{-\frac{1}{2}t^2}u_b \\
		    &+e^{-\frac{1}{2}t^2}\bar{c}\left(\frac{\p}{\p t}\right)\sum_{k=1}^{m}\bar{c}(e_{b,k})\ov{\nabla}^{\ov{S}_b\otimes\ov{E}_b}_{e_{b,k}}u_b-te^{-\frac{1}{2}t^2}u_b+\bar{c}\left(\frac{\p}{\p t}\right)te^{-\frac{1}{2}t^2}u_b \\
		    =& e^{-\frac{1}{2}t^2}\sum_{k=1}^{m}\bar{c}(e_{b,k})\ov{\nabla}^{\ov{S}_b\otimes\ov{E}_b}_{e_{b,k}}u_b+\bar{c}\left(\frac{\p}{\p t}\right)\left(e^{-\frac{1}{2}t^2}\sum_{k=1}^{m}\bar{c}(e_{b,k})\ov{\nabla}^{\ov{S}_b\otimes\ov{E}_b}_{e_{b,k}}u_b\right)
	    \end{aligned}
    \end{equation}
	for a local orthonormal frame $\{e_{b,1},\dots,e_{b,m}\}$ of $T(Z_b\times\{0\})\subset T(Z_b\times\R)$. When traversing counterclockwise, it follows
    \begin{equation}\label{Conterclockwise'}
        \begin{aligned}
		    &\varphi_b\circ D^{E_b}(u_b) \\
		    =& e^{-\frac{1}{2}t^2}\sum_{k=1}^{m}\bar{c}\left(\frac{\p}{\p t}\right)\bar{c}(e_{b,k})\ov{\nabla}^{\ov{S}_b\otimes\ov{E}_b}_{e_{b,k}}u_b-\bar{c}\left(\frac{\p}{\p t}\right)e^{-\frac{1}{2}t^2}\sum_{k=1}^{m}\bar{c}\left(\frac{\p}{\p t}\right)\bar{c}(e_{b,k})\ov{\nabla}^{\ov{S}_b\otimes\ov{E}_b}_{e_{b,k}}u_b \\
		    =& e^{-\frac{1}{2}t^2}\sum_{k=1}^{m}\bar{c}(e_{b,k})\ov{\nabla}^{\ov{S}_b\otimes\ov{E}_b}_{e_{b,k}}u_b+\bar{c}\left(\frac{\p}{\p t}\right)\left(e^{-\frac{1}{2}t^2}\sum_{k=1}^{m}\bar{c}(e_{b,k})\ov{\nabla}^{\ov{S}_b\otimes\ov{E}_b}_{e_{b,k}}u_b\right).
	    \end{aligned}
    \end{equation}
	
	Then (\ref{Clockwise'}) and (\ref{Conterclockwise'}) show that the diagram in Figure \ref{Prop'Diagram1} commutes. Furthermore, by an argument similar to the one used to prove the commutativity of Figure \ref{PropDiagram1'}, this also implies that the diagram in Figure~\ref{Prop'Diagram3} is commutative.
    
	\begin{figure}[hbtp]
		\centering
		\begin{tikzcd}[row sep=large]
			L^2\left(Z_b\times\{0\},\ov{S}_{b,+}\otimes\ov{E}_b\right) \arrow[r, "\varphi_b"] \arrow[d, "P_b"']
			&L^2\left(Z_b\times\R,\ov{S}_b\otimes\ov{E}_b\right) \arrow[d, "Q_b"] \\
			L^2\left(Z_b\times\{0\},\ov{S}_{b,+}\otimes\ov{E}_b\right) \arrow[r, "\varphi_b"]
			&L^2\left(Z_b\times\R,\ov{S}_b\otimes\ov{E}_b\right)
		\end{tikzcd}
		\caption{}\label{Prop'Diagram3}
	\end{figure}
	
	Now by Lemma \ref{IndBunLem'}, it is sufficient to prove that $\varphi_b$ induces an isomorphism
	$$\bar{\varphi}_b\colon\Ker P_b=\Ker D^{E_b}\to\Ker Q_b=\Ker A_b$$
	for all $b\in B$. In fact, $\bar{\varphi}_b$ is the restriction of $\varphi_b$ to $\Ker D^{E_b}$. Since $\varphi_b$ is easily seen to be injective, which implies $\bar{\varphi}_b$ is injective, we only need to prove the surjectivity of $\bar{\varphi}_b$.
	
	Consider the endomorphism $\bar{c}\big(\frac{\p}{\p t}\big)\circ\bar{\tau}$ on the bundle $\ov{S}_b\otimes\ov{E}_b$. Since
	\begin{equation}
		\left(\bar{c}\left(\frac{\p}{\p t}\right)\circ\bar{\tau}\right)^2=\Id,
	\end{equation}
	we obtain another $\Z_2$-grading on $\ov{S}_b\otimes\ov{E}_b$, namely
	\begin{equation}
		\ov{S}_b\otimes\ov{E}_b=\triangle_{b,+}\oplus\triangle_{b,-},
	\end{equation}
	where $\triangle_{b,\pm}$ are the eigenspaces corresponding to the eigenvalues $\pm1$ of $\bar{c}\big(\frac{\p}{\p t}\big)\circ\bar{\tau}$.
	
	It is straightforward to verify the following properties of this grading:
	\begin{itemize}
		\item[$(a)$]
		\begin{equation}
			\begin{aligned}
				\triangle_{b,+}= & \left\{u_b+\bar{c}\left(\frac{\p}{\p t}\right)u_b\colon u_b\in\ov{S}_{b,+}\otimes\ov{E}_b\right\} \\
                =&\left\{v_b-\bar{c}\left(\frac{\p}{\p t}\right)v_b\colon v_b\in\ov{S}_{b,-}\otimes\ov{E}_b\right\}, \\
				\triangle_{b,-}= & \left\{u_b-\bar{c}\left(\frac{\p}{\p t}\right)u_b\colon u_b\in\ov{S}_{b,+}\otimes\ov{E}_b\right\} \\
                =&\left\{v_b+\bar{c}\left(\frac{\p}{\p t}\right)v_b\colon v_b\in\ov{S}_{b,-}\otimes\ov{E}_b\right\};
			\end{aligned}
		\end{equation}
		\item[$(b)$] $\bar{c}\left(\frac{\p}{\p t}\right)$ and $\bar{\tau}$ interchanges $\triangle_{b,+}$ and $\triangle_{b,-}$, whereas $\bar{c}(e)$ for any $e\in T^V(N\times\R)$ orthogonal to $\frac{\p}{\p t}$ preserves this grading;
		\item[$(c)$] The operator $\sum\limits_{k=1}^{m}\bar{c}(e_{b,k})\ov{\nabla}^{\ov{S}_b\otimes\ov{E}_b}_{e_{b,k}}$ preserves this grading, where $\{e_{b,1},\dots,e_{b,m}\}$ is a local orthonormal frame of $T(Z_b\times\{0\})$;
		\item[$(d)$] $\Image\varphi_b\subset L^2\left(Z_b\times\R,\triangle_{b,-}\right)$.
	\end{itemize}
	
	From these properties, we compute for a local orthonormal frame $\big\{e_{b,1},\dots,e_{b,m},\frac{\p}{\p t}\big\}$ of $T(Z_b\times\R)$ that
    \begin{equation}
        \begin{aligned}
		    \left(A_b\right)^2= & \left(\sum_{k=1}^{m}\bar{c}(e_{b,k})\ov{\nabla}^{\ov{S}_b\otimes\ov{E}_b}_{e_{b,k}}+\bar{c}\left(\frac{\p}{\p t}\right)\frac{\p}{\p t}+\bar{\tau}t\right)^2 \\
    		=& \left(\sum_{k=1}^{m}\bar{c}(e_{b,k})\ov{\nabla}^{\ov{S}_b\otimes\ov{E}_b}_{e_{b,k}}\right)^2-\frac{\p^2}{\p t^2}+t^2+\bar{c}\left(\frac{\p}{\p t}\right)\bar{\tau} \\
    		=&\left\{
	    	\begin{aligned}
			    &\left(\sum_{k=1}^{m}\bar{c}(e_{b,k})\ov{\nabla}^{\ov{S}_b\otimes\ov{E}_b}_{e_{b,k}}\right)^2-\frac{\p^2}{\p t^2}+t^2+\Id, & \text{on}\ C^{\infty}\left(Z_b\times\R,\triangle_{b,+}\right), \\
			    &\left(\sum_{k=1}^{m}\bar{c}(e_{b,k})\ov{\nabla}^{\ov{S}_b\otimes\ov{E}_b}_{e_{b,k}}\right)^2-\frac{\p^2}{\p t^2}+t^2-\Id, & \text{on}\ C^{\infty}\left(Z_b\times\R,\triangle_{b,-}\right). \\
		    \end{aligned}			
		    \right.
	    \end{aligned}
    \end{equation}

	Analogous to the proof of Proposition \ref{FamilyProductFormula1}, we conclude that any element in $\Ker A_b=\Ker A_b^2$ is of the form $e^{-\frac{1}{2}t^2}u_b-\bar{c}\left(\frac{\p}{\p t}\right)e^{-\frac{1}{2}t^2}u_b$, where
	$$u_b\in C^{\infty}\left(Z_b\times\{0\},\ov{S}_{b,+}\otimes\ov{E}_b\right)\bigcap\Ker \left(\sum_{k=1}^{m}\bar{c}(e_{b,k})\ov{\nabla}^{\ov{S}_b\otimes\ov{E}_b}_{e_{b,k}}\right).$$
	It lies in the image of $\bar{\varphi}_b=\varphi_b\big|_{\Ker D^{E_b}}$. From this, we conclude that $\bar{\varphi}_b$ is surjective and thus an isomorphism.
	
	By Lemma \ref{IndBunLem'}, we complete the proof of this proposition.
\end{proof}

The same argument as in the proof of Theorem \ref{FamilyProductFormula2} immediately yields the following theorem.

\begin{thm}[Family product formula for odd-dimensional $Z$]\label{FamilyProductFormula2'}
    Let the fiber bundles $(N,g)\to B$ and $\mS\to N$ be defined as before. Suppose the fiber $Z$ of $N\to B$ is odd-dimensional. Consider an Hermitian vector bundle $\big(\mE,h^{\mE}\big)$ over $N$ with an Hermitian connection $\nabla^{\mE}$, and denote the fiber of $\mE$ over $b\in B$ by $E_b$. This gives rise to the corresponding family of twisted Dirac operators $D^{\mE}=\big\{D^{E_b}\big\}_{b\in B}$ and an index bundle $\Ind\big(D^{\mE}\big)\in K^1(B)$.

	Let $\big(\ov{\mE},h^{\ov{\mE}}\big)\to(N\times\R,\bar{g})$ be a bundle with an almost product structure of the type $\big(\mE,h^{\mE}\big)\to(N,g)$, and let $\nabla^{\ov{\mE}}$ be an Hermitian connection on it that agrees with the pullback of $\nabla^{\mE}$ outside a compact subset of $N\times\R$. Denote the spinor bundle of $T^V(N\times\R)$ corresponding to $\bar{g}$ by $\ov{\mS}\to N\times\R$, with an endomorphism $\bar{\tau}$ defined in (\ref{taubar}). Then we obtain a family of twisted Dirac operators $\ov{D}^{\ov{\mE}}=\Big\{\ov{D}^{\ov{E}_b}\Big\}_{b\in B}$ associated with $\ov{\mE}$. In this case, the bundle $\ov{\mS}\otimes\ov{\mE}$ has an almost product structure, and the family of Callias-type operators
	\begin{equation}\label{A_bOdd2}
		A=\left\{A_b\coloneq\ov{D}^{\ov{E}_b}+t\bar{\tau}\right\}_{b\in B}
	\end{equation}
	induces an index bundle $\Ind(A)\in K^1(B)$.
	
	Then we have:
	\begin{equation}
		\Ind(A)=\Ind\left(D^{\mE}\right).
	\end{equation}
\end{thm}

At the end of this section, we present a corollary formulated in the cohomological setting. Let $\pi\colon N\to B$ be a fiber bundle with closed fiber $Z$ and oriented spin vertical tangent bundle $T^VN$. Assume that $\mE$ is an Hermitian vector bundle over $N$. By a slight abuse of notation, denote by $A$ the family of first-order elliptic differential operators defined by \eqref{A_bEven2} when $\dim Z$ is even, and by \eqref{A_bOdd2} when $\dim Z$ is odd. Recall that the Chern character $\ch$ maps $K^*(B)$ to the cohomology group $H^*(B,\C)$. Then, the Atiyah--Singer family index theorem \cite[Theorem 5.1]{ASIV}, together with Theorem \ref{FamilyProductFormula2} and Theorem \ref{FamilyProductFormula2'}, immediately implies the following:
	
\begin{cor}\label{CohomologyCor}
	We have
	\begin{equation}
		\ch(\Ind(A))=\int_{N/B}\widehat{A}\left(T^VN\right)\ch(\mE)\in H^*(B,\C).
	\end{equation}
\end{cor}

\section{Family Band Width Estimate}\label{Section3}
In this section, we apply the family product formula (Theorem \ref{FamilyProductFormula2} and Theorem \ref{FamilyProductFormula2'}) to prove the main result of this article (Theorem \ref{MainTheorem}).

Let $\pi\colon N\to B$ be a fiber bundle with closed fiber $Z$, where the vertical tangent bundle $T^VN$ is oriented and spin. Denote the fiber over $b\in B$ by $Z_b$. As stated in Subsection \ref{Subsection1.2}, given a Euclidean metric $g$ on $T^VN$, we obtain an Hermitian vector bundle $\mS\coloneq S\big(T^VN\big)$ over $N$, which restricts to the spinor bundle $S_b\to Z_b$ associated with the restriction metric $g_b$ of $g$ to $TZ_b$, for each $b\in B$. The Clifford bundle $\Cl(T^VN)$ acts on $\mS$ via the Clifford multiplication $c(\cdot)$. When $\dim Z$ is even, there is a $\Z_2$-grading $\mS=\mS_+\oplus\mS_-$ induced by the endomorphism $\tau$ defined in (\ref{tau}); for odd $\dim Z$, we choose $\mS$ such that $\tau=-\Id$.

For any Hermitian vector bundle $\mE\to N$ with an Hermitian connection $\nabla^{\mE}$, there is a corresponding twisted spinor bundle $\mS\otimes\mE$. When $\dim Z$ is even, this bundle admits a $\Z_2$-grading:
\begin{equation}\label{FamilyTwistedGrading}
	\mS\otimes\mE=\mS_+\otimes\mE\oplus\mS_-\otimes\mE.
\end{equation}
We denote the fiber of $\mE\to N\to B$ over $b\in B$ by $E_b$ as before. Let $D^{\mE}\coloneq\big\{D^{E_b}\big\}_{b\in B}$ be the associated family of twisted Dirac operators, where $D^{E_b}$ is defined as in (\ref{TwistedDiracOp}). For even $\dim Z$, the grading (\ref{FamilyTwistedGrading}) induces a decomposition 
\begin{equation}
	D^{\mE}=
	\begin{pmatrix}
		0 & D^{\mE}_- \\
		D^{\mE}_+ & 0
	\end{pmatrix}.
\end{equation}

As introduced in Section \ref{Section2}, when $\dim Z$ is even, the family $\big\{D^{E_b}_+\big\}_{b\in B}$ defines an index bundle $\Ind\big(D^{\mE}_+\big)\in K^0(B)$. While for odd $\dim Z$, the family $\big\{D^{E_b}\big\}_{b\in B}$ consists of formally self-adjoint operators and defines an index bundle $\Ind\big(D^{\mE}\big)\in K^1(B)$. With slight notational abuse, we summarize both cases by stating that $D^{\mE}$ defines an index bundle $\Ind\big(D^{\mE}\big)\in K^{\dim Z}(B)$.

\begin{defn}\label{InfiniteAhatArea}
	Let $\pi\colon(N,g)\to B$ be a fiber bundle defined as above. It is said to have infinite family $\widehat{A}$-area if for every $\varepsilon>0$, there exists an Hermitian vector bundle $\mE\to N$ with an Hermitian connection $\nabla^{\mE}$, whose natural restrictions to the fibers of $\pi$ are denoted by $\big\{\nabla^{E_b}\big\}_{b\in B}$, such that $\big\|R^{E_b}\big\|_{L^\infty}\leqslant\varepsilon$ for each $b\in B$ and the index bundle $\Ind\big(D^{\mE}\big)\neq0\in K^{\dim Z}(B)$, where
	$$R^{E_b}\coloneq\left(\nabla^{E_b}\right)^2\in\Omega^2(Z_b,\End(E_b))$$
	is the curvature.
\end{defn}

Since $Z$ and $B$ are closed, it is easy to see that the property of having infinite family $\widehat{A}$-area is independent of the choices of $g$. Consequently, we can simply state that the fiber bundle $N\to B$ has infinite family $\widehat{A}$-area.

\begin{ex}\label{InfiniteAhatAeraExample}
	We present several examples of fiber bundles with infinite family $\widehat{A}$-area. In what follows, the notation $\pi\colon N \to B$ always denotes a fiber bundle introduced at the beginning of this section and $Z$ denotes its fiber.
	\begin{itemize}
		\item[$(a)$] Suppose there exists an Hermitian vector bundle $\mE\to N$ equipped with an Hermitian connection $\nabla^{\mE}$ such that the associated family of twisted Dirac operators has non-vanishing index bundle in $K^{\dim Z}(B)$, and for each $b\in B$ the restriction connection $\nabla^{E_b}$ is flat. Then $\pi$ has infinite family $\widehat{A}$-area.
			
		\item[$(b)$] Another sufficient condition for $N\to B$ to have infinite family $\widehat{A}$-area is that the index bundle $\Ind(D)\in K^{\dim Z}(B)$ of the family of untwisted Dirac operators does not vanish. This is a special case of $(a)$, where one takes $\mE$ to be the trivial line bundle and $\nabla^{\mE}$ the trivial connection. An explicit example of such bundles is given in \cite[Theorem 3.2]{BE2023}.
			
		\item[$(c)$] If either the $\widehat{A}$-genus of the fiber $Z$ does not vanish, i.e. $\widehat{A}(Z)\neq0$, or the total space $N$ is spin with $\widehat{A}(N)\neq0$, then the bundle $\pi$ has infinite family $\widehat{A}$-area. Indeed, either condition implies $\Ind(D)\neq0$.
			
		\item[$(d)$] Hanke, Schick and Steimle constructed a fiber bundle $P\to S^{k}$ with fiber $F$ in \cite[Theorem 1.4 and Remark 1.5]{HSS2014}, where $F$ admits a PSC metric, yet $P$ is spin with $\widehat{A}(P)\neq0$. In the case $k=4$, Krannich, Kupers and Randal-Williams provided an explicit construction of a fiber bundle $E\to S^4$ with fiber $\HP^2$, the $2$-dimensional quaternionic projective space, satisfying the same properties \cite[Theorem]{KKR2021}. By part $(c)$, these bundles have infinite family $\widehat{A}$-area.
		
		\item[$(e)$] Consider the product bundle $E\times S^2 \to S^4\times S^2$, obtained by pulling back the bundle $E\to S^4$ in $(d)$ along the projection $S^4\times S^2 \to S^4$. The manifold $S^2$ admits a PSC metric, and hence so does the total space $E\times S^2$. Moreover, since the family index $\Ind(D)$ of the bundle $E\to S^4$ is non-vanishing, the pullback bundle $E\times S^2\to S^4\times S^2$ also has non-vanishing family index, and therefore has infinite family $\widehat{A}$-area.
	\end{itemize}
\end{ex}

There is another important example of a fiber bundle with infinite family $\widehat{A}$-area, whose total space, base, and individual fibers all admit PSC metrics.

\begin{ex}\label{InfiniteAhatAeraExample2}
	Let $p\colon E\to S^4$ be the fiber bundle with fiber $\HP^2$ constructed in \cite{KKR2021}. Denote by $q$ the fiber bundle $\CP^3\to S^4=\HP^1$, which is known to have fiber $S^2$ and structure group $SO(3)$. Pulling back $E$ via $q$, we obtain a fiber bundle $\pi\colon q^*E\to \CP^3$ with fiber $\HP^2$. Since the induced map $q^*\colon H^4(S^4,\C)\to H^4(\CP^3,\C)$ between the cohomology groups is an isomorphism, it follows from the discussion in Example \ref{InfiniteAhatAeraExample} that the index of the family of untwisted Dirac operators associated with $\pi$ does not vanish. Hence, $\pi$ has infinite family $\widehat{A}$-area.

	On the other hand, the total space $q^*E$ admits a PSC metric. Indeed, $q^*E$ can also be regarded as the pullback of the bundle $q$ via $p$ (see Figure \ref{ExampleDiagram1}). Thus, $q^*E\to E$ is an $S^2$-bundle with structure group $SO(3)$, and therefore admits a metric of fiberwise PSC. By applying the O'Neill formulas \cite[Chapter 9]{Besse1987}, one can construct a PSC metric on the total space $q^*E$.

	\begin{figure}[hbtp]
		\centering
		\begin{tikzcd}[row sep=large]
			q^*E \arrow[r] \arrow[d, "\pi"']
			&E \arrow[d, "p"] \\
			\CP^3 \arrow[r, "q"]
			&S^4
		\end{tikzcd}
		\caption{}\label{ExampleDiagram1}
	\end{figure}
\end{ex}

\begin{ex}
	In this example, we construct fiber bundles with infinite family $\widehat{A}$-area, but with vanishing family index of untwisted family Dirac operators. More generally, we prove that if $N\to B$ is a fiber bundle as in Example \ref{InfiniteAhatAeraExample} (b), then for any even-dimensional compactly enlargeable spin manifold $Y$ introduced in \cite[Section 5]{GL1983}, the fiber bundle $N\times Y\to B$ with fiber $Z\times Y$ has infinite family $\widehat{A}$-area. In particular, for $P\to S^k$ in Example \ref{InfiniteAhatAeraExample} (d) and $q^*E\to\C P^3$ in Example \ref{InfiniteAhatAeraExample2}, the fiber bundles $P\times T^{2l}\to S^k$ and $q^*E\times T^{2l}\to\C P^3$ both have infinite family $\widehat{A}$-area, where $T^{2l}$ denotes the $2l$-dimensional torus. It is easy to see that the family indices of untwisted family Dirac operators in these two examples vanish.

	Since $Y$ is compactly enlargeable, for any $\varepsilon>0$, there exists an Hermitian vector bundle $F\to Y$ with an Hermitian connection $\nabla^F$ such that $\|R^F\|_{L^{\infty}}\leqslant\varepsilon$ and
	\begin{equation}
		\int_Y\widehat{A}(TY)\ch(F)\neq0.
	\end{equation}
	Denote by $\pr_2$ the projection from $N\times Y$ to the second factor. We consider the Hermitian vector bundle $\big(\pr_2^*F,\pr_2^*\nabla^F\big)$ over $N\times Y$, and regard it as the twisted part of the spinor bundle. Its curvature, when restricting to each fiber, still has $L^{\infty}$-norm not greater than $\varepsilon$. Moreover, we have
	\begin{equation}
		\ch\left(\Ind\left(D^{\pr_2^*F}\right)\right)=\left(\int_Y\widehat{A}(TY)\ch(F)\right)\cdot\ch\left(\Ind(D)\right)\neq 0.
	\end{equation}
	This implies that $N\times Y\to B$ has infinite family $\widehat{A}$-area.
\end{ex}

Now we prove Theorem \ref{MainTheorem}. Some analytic techniques used in the proof are inspired by \cite[Theorem 1.1]{Yu2024}.

\begin{proof}[Proof of Theorem \ref{MainTheorem}]	
	Consider the cylinder $N\times\R$ and regard it as a natural fiber bundle over $B$. Let $M=N\times[-1,1]$ be viewed as a submanifold of $N\times\R$. Extend the metric $g$ to a complete metric $\bar{g}$ on $N\times\R$ such that:		
	\begin{itemize} 
		\item[$(a)$] $\bar{g}=g$ on $M$; 
		\item[$(b)$] $\bar{g}=g_N\oplus\dd t\otimes\dd t$ on $N\times(-\infty,-2]\bigcup N\times[2,+\infty)$, where $t$ is the parameter of $\R$ and $g_N$ is the restriction of $g$ to $N\times\{0\}$.
	\end{itemize}		
	Clearly, $\mathrm{dist}_{\bar{g}}(x,N\times\{-1\})\geqslant\mathrm{dist}_g(N\times\{-1\},N\times\{1\})$ for any $x\in N\times[1,+\infty)$.
	
	Let $\bar{g}_b$ be the restriction metric of $\bar{g}$ to the fiber over $b\in B$. Since $\bar{g}$ is a product metric outside a compact subset, the fiberwise scalar curvature $k_{\bar{g}_b}$ has a lower bound $\sigma'$, which may not be positive, and this bound can be chosen to be independent of $b$. For simplicity, in the remainder of the proof we set
	$$l\coloneq\mathrm{dist}_g\left(N\times\{-1\},N\times\{1\}\right).$$
	
	We define the signed distance function $\varphi\colon N\times\R\to\R$ by
	\begin{equation}
		\varphi(x)\coloneq 
		\begin{cases}
			-\mathrm{dist}_{\bar{g}}\{x,N\times\{-1\}\}, & x\in N\times(-\infty,-1], \\
			\mathrm{dist}_{\bar{g}}\{x,N\times\{-1\}\}, & x\in N\times[-1,+\infty).
		\end{cases}
	\end{equation}
	It is easy to see that $\varphi$ is a $1$-Lipschitz function.
	
	For every $\varepsilon>0$ (to be specified later), there exists a smooth approximation $\psi$ of $\varphi-\frac{l}{2}$ such that \cite[Theorem 1]{AFMR2007}
	\begin{itemize}
		\item[$(a)$] $\left|\psi(x)-\left(\varphi(x)-\frac{l}{2}\right)\right|\leqslant\frac{\varepsilon}{2},\quad\forall x\in N\times\R$;
		\item[$(b)$] $\psi$ is $(1+\varepsilon)$-Lipschitz.
	\end{itemize}
	Denote by $\psi_b$ the restriction of $\psi$ to $Z_b\times\R$. Then the property $(b)$ implies that $\psi_b$ remains to be a $(1+\varepsilon)$-Lipschitz function on $\left(Z_b\times\R,\bar{g}_b\right)$, and hence its gradient satisfies
	\begin{equation}\label{GradEstimate}
		\left|\grad_{\bar{g}_b}\psi_b\right|\leqslant1+\varepsilon.
	\end{equation}	
	
	We claim that $|\psi(x)| \leqslant \frac{l - \varepsilon}{2}$ implies $x \in M$. In fact,
	\begin{align*}
		|\psi(x)| \leqslant \frac{l - \varepsilon}{2} \quad & \Rightarrow \quad -\frac{l}{2} \leqslant \varphi(x) - \frac{l}{2} \leqslant \frac{l}{2} \\
		& \Rightarrow \quad 0 \leqslant \varphi(x) \leqslant l.
	\end{align*}
	Thus, $x$ belongs to $M$.
	
	After normalization, we may assume $\sigma=n(n-1)$. Our goal is to prove that $l \leqslant \frac{2\pi}{n}$. Suppose on the contrary that $l>\frac{2\pi}{n}$. Then there exist constants $\varepsilon>0$ and $\delta > 0$ such that
	\begin{equation}\label{Condition1}
		\frac{n^2}{4}-\frac{\pi^2(1+\varepsilon)^2}{l^2}>\delta
	\end{equation}
	and
	\begin{equation}\label{Condition2}
		\frac{n\sigma'}{4(n-1)}+\frac{\pi^2(1+\varepsilon)^2}{l^2}\tan^2\left(\frac{\pi(l-2\varepsilon)}{2l}\right)-\varepsilon(1+\varepsilon)>\delta.
	\end{equation}	
	
	Set
	\begin{equation}\label{DefOfr}
		r\coloneq\frac{\pi(1+\varepsilon)}{l}.
	\end{equation}
	We choose a smooth function $\xi\colon\R^2\to\R$ such that
	\begin{equation}\label{DefOfXi}
		\xi(x,y)\coloneq
		\begin{cases}
			\frac{y^2+r^2}{1+\varepsilon}, & |x|\leqslant\frac{l-2\varepsilon}{2}, \\
			\varepsilon, & |x|\geqslant\frac{l-\varepsilon}{2}.
		\end{cases}
	\end{equation}
	We also assume that $\xi$ satisfies the condition
	\begin{equation}\label{xiCondition1}
		0\leqslant(1+\varepsilon)\xi(x,y)\leqslant y^2+r^2,\quad\forall(x,y)\in\R^2,
	\end{equation}
	and that, for any fixed $y\in\R$, the function $\xi(x,y)$ is monotone increasing for $x\in\big[-\frac{l-\varepsilon}{2},-\frac{l-2\varepsilon}{2}\big]$ and monotone decreasing for $x\in\left[\frac{l-2\varepsilon}{2},\frac{l-\varepsilon}{2}\right]$.\footnote{To ensure this, we may need to choose $\varepsilon$ additionally smaller such that $\varepsilon<\frac{r^2}{1+\varepsilon}$.} As a consequence of the properties of $\xi$, we conclude that
	\begin{equation}\label{xiCondition2}
		\xi(x,y)\geqslant\varepsilon,\quad\forall(x,y)\in\R^2.
	\end{equation}
	
	Now we turn to the ordinary differential equation with initial condition
	\begin{equation}\label{ODE}
		\begin{cases}
			f'(s)=\xi(s,f(s)), \\
			f(0)=0.
		\end{cases}
	\end{equation}
	We prove the following lemma.
	\begin{lemma}\label{ODESolutionLem}
		The equation (\ref{ODE}) admits a smooth global solution $f$ on $\R$, satisfying the following properties:
		\begin{itemize}
			\item[$(a)$] $f$ is monotone increasing;
			\item[$(b)$] For $s\in\left[-\frac{l-2\varepsilon}{2},\frac{l-2\varepsilon}{2}\right]$,
			$$f(s)=r\tan\left(\frac{rs}{1+\varepsilon}\right)=r\tan\left(\frac{\pi s}{l}\right);$$
			\item[$(c)$] For $s\geqslant\frac{l-\varepsilon}{2}$, we have
			$$f(s)=\varepsilon s+c_1,$$
			and for $s \leqslant -\frac{l - \varepsilon}{2}$,
			$$f(s)=\varepsilon s+c_2,$$
			where $c_1$ and $c_2$ are constants.
		\end{itemize}
	\end{lemma}
	\begin{proof}		
		First, it is straightforward to verify that the function
		\begin{equation}\label{LocalSolution}
			r\tan\left(\frac{rs}{1+\varepsilon}\right)=r\tan\left(\frac{\pi s}{l}\right)
		\end{equation}
		solves (\ref{ODE}) for $s \in \left[-\frac{l-2\varepsilon}{2},\frac{l-2\varepsilon}{2}\right]$. Next, we demonstrate that this function can be extended to $[0,+\infty)$ as a smooth solution of (\ref{ODE}).
		
		Consider the rectangle
		$$\left[-1,\frac{l-\varepsilon}{2}\right]\times\left[-1,r\tan\left(\frac{\pi(l-\varepsilon)}{2l}\right)\right]\subset\R^2,$$
		drawn with dashed lines in Figure \ref{ODEFig}. The theorem in \cite[Subsection II.7.I]{ODE} yields that there exists a solution $f\colon[0, \alpha]\to\R$ of (\ref{ODE}), where $\alpha\in\left(0,\frac{l-\varepsilon}{2}\right]$, which can be extended to the right such that its graph meets the boundary of the rectangle.
		Moreover, the comparison theorem \cite[Subsection II.9.IX]{ODE}, together with (\ref{xiCondition1}) and (\ref{xiCondition2}), implies
		\begin{equation}
			\varepsilon s\leqslant f(s)\leqslant r\tan\left(\frac{\pi s}{l}\right)
		\end{equation}
		on $[0, \alpha]$. Thus, the graph of $f$ must meet the right edge of the rectangle, as illustrated in Figure \ref{ODEFig}. In other words, we have $\alpha=\frac{l-\varepsilon}{2}$.
		
		\begin{figure}[hbtp]
			\centering
			\begin{tikzpicture}
				[
				declare function={
					f(\s, \r, \l) = \r * tan(pi * \s / \l);
				}
				]
				
				\pgfmathsetmacro{\r}{1}  
				\pgfmathsetmacro{\l}{4}  
				\pgfmathsetmacro{\e}{1/2}  
				
				\pgfmathsetmacro{\X}{(\l/2 - \e/2)}
				\pgfmathsetmacro{\Y}{(\r * tan(pi * deg(\l/2 - \e/2) / \l))}
				
				\begin{axis}[
					axis lines = middle,
					xlabel = $s$,
					ylabel = $y$,
					xmin = -3, xmax = 5,  
					ymin = -2.5, ymax = 6, 
					axis equal, 
					xtick = \empty,  
					ytick = \empty,  
					grid = major,
					grid style = {dashed, gray!30},
					width = 15cm,
					height = 10cm,
					domain = -2:2,         
					samples = 200,         
					restrict y to domain = -10:10,  
					]
					
					\addplot[thick] {f(deg(x), \r, \l)};			
					\addplot[thick] {1/3* x};
					
					\node[font=\small] at (2.1, 0.4) {$\varepsilon s$};
					\node[font=\small] at (2.7, \Y-1.5) {$r\tan\left(\frac{\pi s}{l}\right)$};		
					
					\draw[dashed] (-1,-1) -- (\X,-1);
					\draw[dashed] (\X,-1) -- (\X,\Y); 
					\draw[dashed] (-1,\Y) -- (\X,\Y);
					\draw[dashed] (-1,-1) -- (-1,\Y);			
					
					\node[circle, fill=black, inner sep=0.75pt, label=left:{\small $(-1,-1)$}] at (-1,-1) {};
					\node[circle, fill=black, inner sep=0.75pt, label=right:{\small $\left(\frac{l-\varepsilon}{2},-1\right)$}] at (\X,-1) {};
					\node[circle, fill=black, inner sep=0.75pt, label=left:{\small $\left(-1,r\tan\left(\frac{\pi(l-\varepsilon)}{2l}\right)\right)$}] at (-1,\Y) {};
					\node[circle, fill=black, inner sep=0.75pt, label=right:{\small $\left(\frac{l-\varepsilon}{2},r\tan\left(\frac{\pi(l-\varepsilon)}{2l}\right)\right)$}] at (\X,\Y) {};			
				\end{axis}
			\end{tikzpicture}
			\caption{}\label{ODEFig}
		\end{figure}
		
		On the other hand, it is easy to see that the function $\bar{f}\colon\left[\frac{l-\varepsilon}{2},+\infty\right)\to\R$ defined by
		$$\bar{f}(s)\coloneq\varepsilon\left(s-\frac{l-\varepsilon}{2}\right)+f\left(\frac{l-\varepsilon}{2}\right)$$
		is a solution of $f'(s)=\xi(s,f(s))$ that passes through $\left(\frac{l-\varepsilon}{2},f\left(\frac{l-\varepsilon}{2}\right)\right)$. Thus we conclude from \cite[Subsection II.6.VI (b)]{ODE} that we can glue $f$ and $\bar{f}$ together to obtain a $C^1$ solution on $[0,+\infty)$, which we still denote by $f$.
		
		Since $\xi\equiv\varepsilon$ outside $\left[-\frac{l-\varepsilon}{2},\frac{l-\varepsilon}{2}\right]\times\R$, the theorem in \cite[Subsection II.6.III]{ODE} implies the uniqueness of the $C^1$ solution of (\ref{ODE}) on $[0,+\infty)$. It follows that
		$f$ coincides with the function defined in (\ref{LocalSolution}) on $\left[0,\frac{l-2\varepsilon}{2}\right]$. Moreover, we obtain the smoothness of $f$ from the smoothness of $\xi$ by the corollary in \cite[Subsection III.13.XI]{ODE}. This completes the proof of $(b)$ and $(c)$ in the lemma for $s\in[0,+\infty)$. The case for $s\in(-\infty,0]$ is similar and $f$ is monotone increasing follows directly from $\xi>0$ on $\R^2$. Now we complete the proof of the lemma.
	\end{proof}
	
	We come back to the proof of Theorem \ref{MainTheorem}.
	
	Let $\mE$ be an Hermitian vector bundle over $N$ equipped with an Hermitian connection $\nabla^{\mE}$. This construction yields a family of twisted Dirac operators $D^{\mE}=\big\{D^{E_b}\big\}_{b\in B}$ and an index bundle $\Ind\big(D^{\mE}\big)\in K^{n-1}(B)$, as stated at the beginning of this section.
	
	Denote the projection $N\times\R\to N$ by $\pr$. Then, we have a pullback bundle $\ov{\mE}\coloneq\pr^*\mE$ over $N\times\R$, equipped with the pullback metric and pullback connection. The fiber of $\ov{\mE}$ over $b\in B$ is denoted by $\ov{E}_b$. Following the construction of Subsection \ref{Subsection1.2}, we can define the twisted spinor bundle $\ov{\mS}\otimes\ov{\mE}$, where the Clifford action of $\Cl\left(T^V(N\times\R)\right)$ on it is denoted by $\bar{c}(\cdot)$. If $n$ is odd, we choose $\ov{\mS}$ to satisfy $\bar{\tau}=-\Id$ as before. Additionally, we can define the corresponding family of twisted Dirac operators $\ov{D}^{\ov{\mE}}=\left\{\ov{D}^{\ov{E}_b}\right\}_{b\in B}$, where $\ov{D}^{\ov{E}_b}$ acts on the bundle $\ov{S}_b\otimes\ov{E}_b$.
	
	The Lichnerowicz formula \cite{Lich1963} states that
	\begin{equation}\label{LichFormula}
		\left(\ov{D}^{\ov{E}_b}\right)^2=-\Delta^{\ov{S}_b\otimes\ov{E}_b}+\frac{k_{\bar{g}_b}}{4}+\mR^{\ov{E}_b}
	\end{equation}
	for each $b\in B$, where $\Delta^{\ov{S}_b\otimes\ov{E}_b}$ is the Bochner Laplacian associated with the connection $\nabla^{\ov{S}_b\otimes\ov{E}_b}$ defined analogously to $\nabla^{S_b\otimes E_b}$ in (\ref{TwistedDiracOp}) and
	\begin{equation}
		\mR^{\ov{E}_b}\coloneq\frac{1}{2}\sum_{k,l=1}^{n}\bar{c}(e_{b,k})\bar{c}(e_{b,l})R^{\ov{E}_b}\left(e_{b,k},e_{b,l}\right)
	\end{equation}
	in a local orthonormal frame $\{e_{b,1},\dots,e_{b,n}\}$ of $T(Z_b\times\R)$. By Cauchy-Schwarz inequality, we deduce that
	\begin{equation}
		\left\langle\ov{D}^{\ov{E}_b}u,\ov{D}^{\ov{E}_b}u\right\rangle\leqslant n\left\langle\nabla^{\ov{S}_b\otimes\ov{E}_b}u,\nabla^{\ov{S}_b\otimes\ov{E}_b}u\right\rangle
	\end{equation}
	holds pointwisely for any $u\in C_0^{\infty}(Z_b\times\R,\ov{S}_b\otimes\ov{E}_b)$. Integrating it and applying divergence theorem yields
	\begin{equation}\label{IntegralCauchy}
		\int_{Z_b\times\R}\Big|\ov{D}^{\ov{E}_b} u\Big|^2\,\dd\vol_{\bar{g}_b}\leqslant n\int_{Z_b\times\R}\left\langle-\Delta^{\ov{S}_b\otimes\ov{E}_b}u,u\right\rangle\,\dd\vol_{\bar{g}_b}.
	\end{equation}
	From (\ref{LichFormula}) and (\ref{IntegralCauchy}), we obtain the Friedrich's inequality \cite{Fried1980}
	\begin{equation}\label{FriedrichInequality}
		\int_{Z_b\times\R}\Big|\ov{D}^{\ov{E}_b} u\Big|^2\,\dd\vol_{\bar{g}_b}\geqslant \int_{Z_b\times\R}\left(\frac{n}{4(n-1)}k_{\bar{g}_b}|u|^2+\frac{n}{n-1}\left\langle\mR^{\ov{E}_b}u,u\right\rangle\right)\,\dd\vol_{\bar{g}_b},
	\end{equation}
	which holds for each $b\in B$.
	
	Note that $R^{\ov{E}_b}=\pr^*R^{E_b}$. Since $N$ has infinite family $\widehat{A}$-area, we can choose $\mE$ and $\nabla^{\mE}$ such that
	the index bundle $\Ind\big(D^{\mE}\big)\neq0\in K^{n-1}(B)$ and
	\begin{equation}\label{CurvatureCondition}
		\frac{n}{n-1}\left\|\mR^{\ov{E}_b}\right\|_{L^{\infty}}\leqslant\frac{\delta}{2}
	\end{equation}
	with respect to the metric $\bar{g}_b$ for all $b\in B$, where $\delta$ is the constant appearing in (\ref{Condition1}) and (\ref{Condition2}).
	
	The remainder of the argument proceeds according to the parity of $\dim Z$. We begin with the case where $\dim Z$ is even.
	
	Consider the family of first-order elliptic differential operators
	\begin{equation}
		P=\Big\{P_b\coloneq\ov{D}^{\ov{E}_b}+\sqrt{-1}f(\psi_b)\Id\Big\}_{b\in B},
	\end{equation}
	where $f$ is the global solution of (\ref{ODE}) obtained in Lemma \ref{ODESolutionLem}. We compute that
	\begin{equation}\label{Pb*Pb}
		P_b^*P_b=\left(\ov{D}^{\ov{E}_b}\right)^2+f^2(\psi_b)\Id+\sqrt{-1}f'(\psi_b)c\left(\grad_{\bar{g}_b}\psi_b\right)
	\end{equation}
	and\begin{equation}\label{PbPb*}
		P_bP_b^*=\left(\ov{D}^{\ov{E}_b}\right)^2+f^2(\psi_b)\Id-\sqrt{-1}f'(\psi_b)c\left(\grad_{\bar{g}_b}\psi_b\right).
	\end{equation}
	From the construction of $\psi_b$ and $(c)$ in Lemma \ref{ODESolutionLem}, we know that outside a compact subset,
	\begin{equation}
		\left|f^2(\psi_b)\Id\pm \sqrt{-1}f'(\psi_b)c\left(\grad_{\bar{g}_b}\psi_b\right)\right|\geqslant \varepsilon^2\psi_b^2-c
	\end{equation}
	is uniformly positive, where $c$ is a constant. 
	
	Since $f$ has uniformly bounded derivative and $\psi$ is a Lipschitz function on $N\times\R$, we conclude that $f(\psi_b)\to f(\psi_{b_0})$ uniformly as $b\to b_0$ in a local trivialization of $N\times\R$ and $\ov{\mE}$. Therefore, we obtain an estimate analogous to \eqref{GraphEstimateFinal} for $P_b$. Since $\ov{\mS}\otimes\ov{\mE}$ has an almost product structure, the discussion in Subsection \ref{Subsection2.1} applies similarly, and we obtain an index bundle $\Ind(P)\in K^0(B)$.
	
	We first prove that $P_b$ is invertible for each $b$. As a consequence, the index bundle $\Ind(P)$ vanishes in $K^0(B)$.
	
	Since $f$ solves (\ref{ODE}), we conclude from (\ref{Pb*Pb}) that
	\begin{equation}\label{Pb*Pb2}
		P_b^*P_b=\left(\ov{D}^{\ov{E}_b}\right)^2+f^2(\psi_b)\Id+\sqrt{-1}\xi(\psi_b,f(\psi_b))c\left(\grad_{\bar{g}_b}\psi_b\right).
	\end{equation}
	Pick an arbitrary $u\in C_0^{\infty}(Z_b\times\R,\ov{S}_b\otimes\ov{E}_b)$. Then (\ref{GradEstimate}), (\ref{FriedrichInequality}), (\ref{CurvatureCondition}) and (\ref{Pb*Pb2}) yield
	{\small
		\begin{equation}\label{Int|Pbu|^2}
			\int_{Z_b\times\R}\left|P_bu\right|^2\,\dd\vol_{\bar{g}_b}\geqslant\int_{Z_b\times\R}\left(\frac{n}{4(n-1)}k_{\bar{g}_b}+f^2(\psi_b)-\xi(\psi_b,f(\psi_b))(1+\varepsilon)-\frac{\delta}{2}\right)|u|^2\,\dd\vol_{\bar{g}_b}.
		\end{equation}
	}
	
	On $Z_b\times[-1,1]$, the scalar curvature $k_{\bar{g}_b}\geqslant n(n-1)$. Together with (\ref{Condition1}), (\ref{DefOfr}) and (\ref{xiCondition1}), we conclude
	\begin{equation}\label{Fun>delta1}
		\frac{n}{4(n-1)}k_{\bar{g}_b}+f^2(\psi_b)-\xi(\psi_b,f(\psi_b))(1+\varepsilon)-\frac{\delta}{2} \geqslant \frac{n^2}{4}-r^2-\frac{\delta}{2} >\frac{\delta}{2}.
	\end{equation}
	
	Outside $Z_b\times[-1,1]$, the function $\psi_b$ satisfies $|\psi_b|>\frac{l-\varepsilon}{2}$. Thus, $\xi(\psi_b,f(\psi_b))=\varepsilon$ by (\ref{DefOfXi}). Moreover, $(a)$ and $(b)$ in Lemma \ref{ODESolutionLem} yield
	\begin{equation}\label{fPsib>}
		f^2(\psi_b)>f^2\left(\frac{l-2\varepsilon}{2}\right) =r^2\tan^2\left(\frac{\pi(l-2\varepsilon)}{2l}\right).
	\end{equation}
	In this case, we obtain via (\ref{Condition2}), (\ref{DefOfr}) and (\ref{fPsib>}) that
    \begin{equation}\label{Fun>delta2}
        \begin{aligned}
            &\frac{n}{4(n-1)}k_{\bar{g}_b}+f^2(\psi_b)-\xi(\psi_b,f(\psi_b))(1+\varepsilon)-\frac{\delta}{2} \\
		    > & \frac{n\sigma'}{4(n-1)}+r^2\tan^2\left(\frac{\pi(l-2\varepsilon)}{2l}\right)-\varepsilon(1+\varepsilon)-\frac{\delta}{2} \\
		    > & \frac{\delta}{2}.
        \end{aligned}
    \end{equation}
	
	Combining (\ref{Fun>delta1}) and (\ref{Fun>delta2}) with (\ref{Int|Pbu|^2}), we finally obtain
	\begin{equation}
		\int_{Z_b\times\R}\left|P_bu\right|^2\,\dd\vol_{\bar{g}_b}\geqslant\frac{\delta}{2}\int_{Z_b\times\R}|u|^2\,\dd\vol_{\bar{g}_b}.
	\end{equation}
	Therefore, $\Ker(P_b)=\{0\}$. A similar argument yields $\Ker(P_b^*)=\{0\}$. Hence, the operator $\frac{P_b}{\sqrt{\Id+P_b^*P_b}}$ is invertible for each $b$. We have thus proven that $\Ind(P)=0$.
	
	On the other hand, denote
    \begin{equation}
        Q_b(s)\coloneq\ov{D}^{\ov{E}_b}+\sqrt{-1}\left(sf(\psi_b)+(1-s)t\right)\Id
    \end{equation}
	for $s\in[0,1]$, where $t$ is the restriction of the projection $N\times\R\to\R$ to $Z_b\times\R$. It is straightforward to verify that for each $s\in[0,1]$, the family $\Big\{\frac{Q_b(s)}{\sqrt{\Id+Q_b(s)^*Q_b(s)}}\Big\}_{b\in B}$ is a continuous family of Fredholm operators, and that it varies continuously with respect to $s$. Thus we obtain a continuous family of index bundles $\{\Ind(Q(s))\}_{s\in[0,1]}\subset K^0(B)$. It follows that
	\begin{equation}\label{IndexBundleEq}
		\Ind(P)=\Ind(Q(1))=\Ind(Q(0))=\Ind\Big(\ov{D}^{\ov{\mE}}+\sqrt{-1}t\Id\Big).
	\end{equation}
	However, Theorem \ref{FamilyProductFormula2} states that
	$$\Ind\Big(\ov{D}^{\ov{\mE}}+\sqrt{-1}t\Id\Big)=\Ind\left(D^{\mE}_+\right),$$
	which by assumption is nonzero. This contradicts $\Ind(P)=0$. Therefore, we must have $l\leqslant\frac{2\pi}{n}$.
	
	We now turn to the case where $Z$ is odd-dimensional.
	
	In this situation, we set
	\begin{equation}
		\widetilde{P}=\left\{\widetilde{P}_b\coloneq\ov{D}^{\ov{E}_b}+f(\psi_b)\bar{\tau}\right\}_{b\in B},
	\end{equation}
	which is a smooth family of formally self-adjoint first-order elliptic differential operators. A direct computation shows that
	\begin{equation}\label{WidetildePb^2}
		\left(\widetilde{P}_b\right)^2=\left(\ov{D}^{\ov{E}_b}\right)^2+f^2(\psi_b)\Id+f'(\psi_b)c\left(\grad_{\bar{g}_b}\psi_b\right)\bar{\tau}.
	\end{equation}
	As in the even-dimensional case, the family $\widetilde{P}$ defines an index bundle $\Ind\big(\widetilde{P}\big)\in K^1(B)$. Moreover, by an estimate analogous to that in the even-dimensional case, one can show that each $\widetilde{P}_b$ is invertible. Consequently, the operators $\Id\cos\pi\theta\pm\sqrt{-1}\widetilde{P}\sin\pi\theta$ are invertible for all $\theta\in[0,1]$ and $b\in B$. Therefore, $\Ind\big(\widetilde{P}\big)=0$.
	
	On the other hand, consider the family
	\begin{equation}
		\widetilde{Q}_b(s)\coloneq\ov{D}^{\ov{E}_b}+\left(sf(\psi_b)+(1-s)t\right)\bar{\tau}
	\end{equation}
	for $s\in[0,1]$. Analogous to (\ref{IndexBundleEq}), we have
	\begin{equation}
		\Ind\left(\widetilde{P}\right)=\Ind\left(\widetilde{Q}(1)\right)=\Ind\left(\widetilde{Q}(0)\right)= \Ind\left(\ov{D}^{\ov{\mE}}+t\bar{\tau}\right).
	\end{equation}
	However, by Theorem \ref{FamilyProductFormula2'} and our assumption, this index is nonzero, leading to a contradiction. We therefore conclude that $l\leqslant\frac{2\pi}{n}$.
	
	Having addressed both cases when $Z$ is of even and odd dimension, we have thus completed the proof of Theorem \ref{MainTheorem}.
\end{proof}

\appendix
\section{Proofs of Lemma \ref{IndBunLem} and Lemma \ref{IndBunLem'}}
	
\begin{proof}[Proof of Lemma \ref{IndBunLem}]
	Recall that for each $b\in B$ we have a commutative diagram as shown in Figure \ref{IndBunLemFig1} (also refer to the middle square of the diagram in Figure \ref{IndBunLemFig2}), and two induced isomorphisms $\mu_b\colon\Ker P_b\to\Ker Q_b$ and $\eta_b\colon\Coker P_b\to\Coker Q_b$ via diagram chasing.
		
	Note that $\mu_b$ is simply the restriction $\varphi_b\big|_{\Ker P_b}$. If we denote by $p\colon H_{2,b}'\to\Image Q_b$ and $p^\perp\colon H_{2,b}'\to(\Image Q_b)^\perp$ the orthogonal projections, then $\eta_b$ can be identified with the composition
	\begin{equation}
		\eta_b=p^\perp\circ\psi_b\big|_{(\Image P_b)^\perp}.
	\end{equation}
	See the commutative diagram in Figure \ref{IndBunLemFig2}.
		
	\begin{figure}[hbtp]
		\centering
		\begin{tikzcd}
			\Ker P_b \arrow[r, hook] \arrow[d, "\mu_b", "\rotatebox{90}{$\sim$}"']
			&H_{1,b} \arrow[d, "\varphi_b"] \arrow[r, "P_b"]
			&H_{1,b}' \arrow[d, "\psi_b"] \arrow[r, two heads]
			&\Coker P_b\cong(\Image P_b)^{\perp} \arrow[d, "\eta_b", "\rotatebox{90}{$\sim$}"', xshift=5ex]
			\\
			\Ker Q_b \arrow[r, hook]
			&H_{2,b} \arrow[r, "Q_b"]
			&H_{2,b}' \arrow[r, two heads]
			&\Coker Q_b\cong(\Image Q_b)^{\perp}
		\end{tikzcd}
		\caption{}\label{IndBunLemFig2}
	\end{figure}
		
	By Atiyah and Singer's construction \cite[Proposition 2.2]{ASIV}, there exist $l\in\N$ and global sections $\{w_{1,1}',\dots,w_{1,l}'\}$ of $\mH_1'\to B$, such that the operator
	\begin{align*}
		\widetilde{P}_b\colon\quad H_{1,b}\oplus\C^l\quad\, & \to\qquad\quad H_{1,b}' \\
		(u_{1,b};z_1,\dots,z_l) & \mapsto P_b(u_{1,b})+\sum_{k=1}^{l}z_kw_{1,k}'(b)
	\end{align*}
	is surjective for each $b\in B$. Note that $\big\{\widetilde{P}_b\big\}_{b\in B}$ remains to be a continuous family of Fredholm operators. Consequently, $\Ker \widetilde{P}_b$ has finite constant rank and forms a vector bundle $\Ker\widetilde{P}$ over $B$. The index bundle $\Ind(P)$ is then defined by
	\begin{equation}\label{IndBunLemIndP}
		\Ind(P)=\left[\Ker\widetilde{P}\right]-\left[\C^l\right],
	\end{equation}
	where $\left[\C^l\right]$ denotes the element in $K^0(B)$ represented by the trivial bundle of rank $l$.
		
	Define the operator
	\begin{align*}
		\widetilde{Q}_b\colon\quad H_{2,b}\oplus\C^l\quad\, & \to\qquad\qquad H_{2,b}' \\
		(v_{2,b};z_1,\dots,z_l) & \mapsto Q_b(v_{2,b})+\sum_{k=1}^{l}z_k\psi_b\left(w_{1,k}'(b)\right).
	\end{align*}
	We first prove $\widetilde{Q}_b$ is surjective. For any $v_{2,b}'\in(\Image Q_b)^{\perp}$. The surjectivity of $\eta_b$ guarantees the existence of $u_{1,b}'\in(\Image P_b)^{\perp}$ satisfying $\eta_b(u_{1,b}')=v_{2,b}'$, i.e.
	\begin{equation}
		p^{\perp}\circ\psi_b(u_{1,b}')=v_{2,b}'.
	\end{equation}
	Thus $v_{2,b}'-\psi_b(u_{1,b}')\in\Image Q_b$, and we can set
	\begin{equation}\label{Qbv2b}
		v_{2,b}'-\psi_b(u_{1,b}')=Q_b(v_{2,b}).
	\end{equation}
	Since $\widetilde{P}_b$ is surjective, there exists $(u_{1,b};z_1,\dots,z_l)\in H_{1,b}\oplus\C^l$ such that
	\begin{equation}\label{u1b'}
		u_{1,b}'=P_b(u_{1,b})+\sum_{k=1}^{l}z_kw_{1,k}'(b).
	\end{equation}
	Substituting (\ref{u1b'}) into (\ref{Qbv2b}), we obtain
	\begin{equation}	
		\begin{aligned}
			v_{2,b}' & =\psi_b\circ P_b(u_{1,b})+\sum_{k=1}^{l}z_k\psi_b\left(w_{1,k}'(b)\right)+Q_b(v_{2,b}) \\
			& =\widetilde{Q}_b\left(\varphi_b(u_{1,b})+v_{2,b};z_1,\dots,z_l\right) \\
			& \in\Image\widetilde{Q}_b,
		\end{aligned}
	\end{equation}
	i.e., $(\Image Q_b)^{\perp}\subseteq\Image\widetilde{Q}_b$. Furthermore, obviously $\Image Q_b$ is contained in $\Image\widetilde{Q}_b$. This proves the surjectivity of $\widetilde{Q}_b$ and consequently we obtain, again by \cite[Proposition 2.2]{ASIV}, that
	\begin{equation}\label{IndBunLemIndQ}
		\Ind(Q)=\left[\Ker\widetilde{Q}\right]-\left[\C^l\right].
	\end{equation}
		
	Denote the map $(\varphi_b,\Id)\colon H_{1,b}\oplus\C^l\to H_{2,b}\oplus\C^l$ by $\varphi_b'$. A direct verification shows that $\widetilde{Q}_b\circ\varphi_b'=\psi_b\circ\widetilde{P}_b$, hence $\varphi_b'$ induces $\mu_b'\coloneq\varphi_b'\big|_{\Ker\widetilde{P}_b}\colon\Ker\widetilde{P}_b\to\Ker\widetilde{Q}_b$ via diagram chasing. We obtain a commutative diagram shown in Figure \ref{IndBunLemFig3}.
		
	\begin{figure}[hbtp]
		\centering
		\begin{tikzcd}[row sep=large, column sep=large]
			\Ker\widetilde{P}_b \arrow[r, hook] \arrow[d, "\mu_b'"']
			&H_{1,b}\oplus\C^l \arrow[d, "\varphi_b'"'] \arrow[r, two heads, "\widetilde{P}_b"]
			&H_{1,b}' \arrow[d, "\psi_b"]
			\\
			\Ker\widetilde{Q}_b \arrow[r, hook]
			&H_{2,b}\oplus\C^l \arrow[r, two heads, "\widetilde{Q}_b"]
			&H_{2,b}'
		\end{tikzcd}
		\caption{}\label{IndBunLemFig3}
	\end{figure}
		
	We aim to prove that $\mu_b'$ is an isomorphism. For injectivity, let $(u_{1,b};z_1,\dots,z_l)\in\Ker\mu_b'\subseteq\Ker\widetilde{P}_b$. From the equality
	\begin{equation}
		0=\mu_b'(u_{1,b};z_1,\dots,z_l)=\left(\varphi_b(u_{1,b});z_1,\dots,z_l\right),
	\end{equation}
	we deduce that $z_1=\cdots=z_l=0$ and $\varphi_b(u_{1,b})=0$. Then $\widetilde{P}_b(u_{1,b};z_1,\dots,z_l)=0$ yields $u_{1,b}\in\Ker P_b$. Note that $\mu_b=\varphi_b\big|_{\Ker P_b}$ is injective by assumption, hence $u_{1,b}=0$. This proves $\mu_b'$ is injective.
		
	We now prove the surjectivity of $\mu_b'$. Consider an arbitrary element $(v_{2,b};z_1,\dots,z_l)\in\Ker\widetilde{Q}_b$, which by definition satisfies
	\begin{equation}\label{IndBunLemEq1}
		Q_b(v_{2,b})+\sum_{k=1}^{l}z_k\psi_b\left(w_{1,k}'(b)\right)=0.
	\end{equation}
	Thus $\psi_b\left(\sum\limits_{k=1}^{l}z_kw_{1,k}'(b)\right)\in\Image Q_b$. Suppose $\sum\limits_{k=1}^{l}z_kw_{1,k}'(b)=\bar{u}_{1,b}'+\tilde{u}_{1,b}'$, where $\bar{u}_{1,b}'\in\Image P_b$ and $\tilde{u}_{1,b}'\in(\Image P_b)^{\perp}$. Due to the commutativity of the middle square in Figure \ref{IndBunLemFig2}, $\psi_b$ maps $\Image P_b$ to $\Image Q_b$. Then
	\begin{equation}
		\psi_b\left(\tilde{u}_{1,b}'\right)=\psi_b\left(\sum\limits_{k=1}^{l}z_kw_{1,k}'(b)\right)-\psi_b\left(\bar{u}_{1,b}'\right)\in\Image Q_b.
	\end{equation}
	In other words, we have $p^{\perp}\circ\psi_b\big(\tilde{u}_{1,b}'\big)=0$. Since $\eta_b=p^{\perp}\circ\psi_b\big|_{(\Image P_b)^{\perp}}$ is injective, we conclude $\tilde{u}_{1,b}'=0$, proving $\sum\limits_{k=1}^{l}z_kw_{1,k}'(b)\in\Image P_b$. Set
	\begin{equation}\label{IndBunLemEq2}
		\sum_{k=1}^{l}z_kw_{1,k}'(b)=P_b\left(\bar{u}_{1,b}\right).
	\end{equation}
	Substituting (\ref{IndBunLemEq2}) into (\ref{IndBunLemEq1}) yields
	\begin{equation}
		Q_b(v_{2,b})+Q_b\circ\varphi_b(\bar{u}_{1,b})=0,
	\end{equation}
	or equivalently, $v_{2,b}+\varphi_b(\bar{u}_{1,b})\in\Ker Q_b$. Then due to the surjectivity of $\mu_b$, there exists $\tilde{u}_{1,b}\in\Ker P_b$ such that $\varphi_b(\tilde{u}_{1,b})=v_{2,b}+\varphi_b(\bar{u}_{1,b})$. As a result, we obtain
	\begin{equation}\label{IndBunLemEq3}
		(v_{2,b};z_1,\dots,z_l)=\varphi_b'(\tilde{u}_{1,b}-\bar{u}_{1,b};z_1,\dots,z_l).
	\end{equation}
	Moreover, from (\ref{IndBunLemEq2}) we compute
	\begin{equation}\label{IndBunLemEq4}
		\widetilde{P}_b(\tilde{u}_{1,b}-\bar{u}_{1,b};z_1,\dots,z_l)=-P_b(\bar{u}_{1,b})+\sum_{k=1}^{l}z_kw_{1,k}'(b)=0.
	\end{equation}
	It follows from (\ref{IndBunLemEq3}) and (\ref{IndBunLemEq4}) that $(v_{2,b};z_1,\dots,z_l)\in\Image\mu_b'$, completing the proof of the surjectivity of $\mu_b'$.
		
	In summary, we have shown that for each $b\in B$, the map $\mu_b'\colon\Ker\widetilde{P}_b\to\Ker\widetilde{Q}_b$ is an isomorphism. Note that the family $\mu'\coloneq\big\{\mu_b'\big\}_{b\in B}$ is just the restriction of the continuous family $\varphi'\coloneq\big\{\varphi_b'\big\}_{b\in B}$ of linear operators to the subbundle $\Ker\widetilde{P}$, hence yields a bundle isomorphism from $\Ker\widetilde{P}$ to $\Ker\widetilde{Q}$. Combining this with (\ref{IndBunLemIndP}) and (\ref{IndBunLemIndQ}), we finally complete the proof of this lemma.
\end{proof}
	
\begin{proof}[Proof of Lemma \ref{IndBunLem'}]
	First, motivated by \cite[Theorem B]{AS1969}, we consider the following diagram. For each $s \in [0,1]$, its commutativity follows immediately from that of Figure \ref{IndBunLem'Fig1}.
		
	\begin{figure}[hbtp]
		\centering
		\begin{tikzcd}[row sep=large, column sep=huge]
			H_{1,b} \arrow[d, "\varphi_b"'] \arrow[rr, "\Id\cos\pi s+\sqrt{-1}P_b\sin\pi s"]
			&&H_{1,b} \arrow[d, "\varphi_b"]
			\\
			H_{2,b} \arrow[rr, "\Id\cos\pi s+\sqrt{-1}Q_b\sin\pi s"]
			&&H_{2,b}
		\end{tikzcd}
		\caption{}\label{IndBunLem'Fig2}
	\end{figure}
		
	From the identity
	\begin{equation}
		\left(\Id\cos\pi s\mp\sqrt{-1}P_b\sin\pi s\right)\left(\Id\cos\pi s\pm\sqrt{-1}P_b\sin\pi s\right)=\Id\cos^2\pi s+P_b^2\sin^2\pi s,
	\end{equation}
	we see that $\Id\cos\pi s+\sqrt{-1}P_b\sin\pi s$ is invertible for $s\neq\frac{1}{2}$, with the same result holding when replacing $P_b$ by $Q_b$. When $s = \frac{1}{2}$, Figure \ref{IndBunLem'Fig2} reduces to Figure \ref{IndBunLem'Fig1} up to multiplication by the constant $\sqrt{-1}$.
		
	Since $P_b$ and $Q_b$ are self-adjoint, the induced map of $\varphi_b$ from $\Coker\big(\sqrt{-1}P_b\big)$ to $\Coker\big(\sqrt{-1}Q_b\big)$ by Figure \ref{IndBunLem'Fig2} at $s=\frac{1}{2}$ can be identified with $\mu_b\colon\Ker P_b\to\Ker Q_b$, which is an isomorphism. Consequently, for each $s\in[0,1]$, $\varphi_b$ induces isomorphisms:
	\begin{align*}
		\Ker\left(\Id\cos\pi s+\sqrt{-1}P_b\sin\pi s\right) & \to\Ker\left(\Id\cos\pi s+\sqrt{-1}Q_b\sin\pi s\right), \\
		\Coker\left(\Id\cos\pi s+\sqrt{-1}P_b\sin\pi s\right) & \to\Coker\left(\Id\cos\pi s+\sqrt{-1}Q_b\sin\pi s\right).
	\end{align*}

    As discussed in Subsection \ref{Subsection1.1}, we may regard
	$$\left\{\Id\cos\pi s+\sqrt{-1}P_b\sin\pi s\right\}_{s\in[0,1],b\in B},\quad \left\{\Id\cos\pi s+\sqrt{-1}Q_b\sin\pi s\right\}_{s\in[0,1],b\in B}$$
	as families of Fredholm operators on the suspension $\Sigma B$ of $B$. Applying Lemma \ref{IndBunLem} to the commutative diagram in Figure \ref{IndBunLem'Fig2} yields
	\begin{equation}\label{IndBunLem'IndPQ}
        \begin{aligned}
            &\Ind\left(\left\{\Id\cos\pi s+\sqrt{-1}P_b\sin\pi s\right\}_{s\in[0,1],b\in B}\right)= \\
            &\hspace{3cm}\Ind\left(\left\{\Id\cos\pi s+\sqrt{-1}Q_b\sin\pi s\right\}_{s\in[0,1],b\in B}\right)\in K^0(\Sigma B).
        \end{aligned}
	\end{equation}
    Then by definition, this implies $\Ind(P)=\Ind(Q)\in K^1(B)$. Now we complete the proof.
\end{proof}

\noindent{\bf Acknowledgments.}
The author would like to thank Professor Weiping Zhang for his guidance and many invaluable suggestions throughout this work. The author also thanks Professor Xiaonan Ma, Professor Guangxiang Su, and his colleague Hengyu Chen for many enlightening discussions. 

This work was partially supported by the National Key R\&D Program of China (No. 2024YFA1013202), NSFC Grant No. 11931007 and the Nankai Zhide Foundation.


\begin{thebibliography}{99}	
	\bibitem{Ang1993} N. Anghel, An abstract index theorem on noncompact Riemannian manifolds. {\it Houston J. Math.} 19 (1993), 223--237.
	
	\bibitem{Ang19932} N. Anghel, On the index of Callias-type operators. {\it Geom. Funct. Anal.} 3 (1993), 431--438.
	
	\bibitem{AS1969} M. F. Atiyah and I. M. Singer, Index theory for skew-adjoint Fredholm operators. {\it Publ. Math. Inst. Hautes {\'E}tudes Sci.} 37 (1969), 5--26.
	
	\bibitem{ASIV} M. F. Atiyah and I. M. Singer, The index of elliptic operators: IV. {\it Ann. Math.} 93 (1971), 119--138.
	
	\bibitem{AFMR2007} D. Azagra, J. Ferrera, F. L{\'o}pez-Mesas, and Y. Rangel, Smooth approximation of Lipschitz functions on Riemannian manifolds. {\it J. Math. Anal. Appl.} 326 (2007), 1370--1378.

	\bibitem{Besse1987} A. L. Besse, {\it Einstein Manifolds}. Ergebnisse der Mathematik und ihrer Grenzgebiete, vol. 10. Springer-Verlag, Berlin, 1987.
	
	\bibitem{BE2023} B. Botvinnik and J. Ebert, Positive Scalar Curvature and Homotopy Theory. {\it Perspectives in Scalar Curvature}, Vol. 2, eds. by M. Gromov and H. B. Lawson. World Sci. Publ., Hackensack, NJ, 2023, 83--157.
	
	\bibitem{Ca1978} C. Callias, Axial anomalies and index theorems on open spaces. {\it Commun. Math. Phys.} 62 (1978), 213--234.
	
	\bibitem{Ce2020} S. Cecchini, A long neck principle for Riemannian spin manifolds with positive scalar curvature. {\it Geom. Funct. Anal.} 30 (2020), 1183--1223.
	
	\bibitem{CZ2024} S. Cecchini and R. Zeidler, Scalar and mean curvature comparison via the Dirac operator. {\it Geom. Topol.} 28 (2024), 1167--1212.

	\bibitem{Ebert2017} J. Ebert, The two definitions of the index difference. {\it Trans. Amer. Math. Soc.} 369 (2017), no. 10, 7469--7507.
	
	\bibitem{Fried1980} Th. Friedrich, Der erste Eigenwert des Dirac-Operators einer kompakten, Riemannschen Mannigfaltigkeit nichtnegativer Skalarkrümmung. {\it Math. Nachr.} 97 (1980), 117--146.
	
	\bibitem{GJ1987} J. Glimm and A. Jaffe, {\it Quantum Physics}. Second edition. Springer-Verlag, New York, 1987.
	
	\bibitem{Gro2018} M. Gromov, Metric inequalities with scalar curvature. {\it Geom. Funct. Anal.} 28 (2018), 645--726.
	
	\bibitem{Gro2023} M. Gromov, Four lectures on scalar curvature. {\it Perspectives in Scalar Curvature}, Vol. 1, eds. by M. Gromov and H. B. Lawson. World Sci. Publ., Hackensack, NJ, 2023, 1--514.
	
	\bibitem{GL1983} M. Gromov and H. B. Lawson, Positive scalar curvature and the Dirac operator on complete Riemannian manifolds. {\it Publ. Math. Inst. Hautes {\'E}tudes Sci.} 58 (1983), 83--196.
	
	\bibitem{GXY2020} H. Guo, Z. Xie, and G. Yu, Quantitative K-Theory, positive scalar curvature, and bandwidth. {\it Perspectives in Scalar Curvature}, Vol. 2, eds. by M. Gromov and H. B. Lawson. World Sci. Publ., Hackensack, NJ, 2023, 763--798.
	
	\bibitem{HSS2014} B. Hanke, T. Schick and W. Steimle, The space of metrics of positive scalar curvature. {\it Publ. Math. Inst. Hautes {\'E}tudes Sci.} 120 (2014), 335--367.
		
	\bibitem{KKR2021} M. Krannich, A. Kupers and O. Randal-Williams, An $\HP^2$-bundle over $S^4$ with nontrivial $\widehat{A}$-genus. {\it C. R. Math. Acad. Sci. Paris} 359 (2021), 149--154.
	
	\bibitem{LM1989} H. B. Lawson and M.-L. Michelsohn, {\it Spin Geometry}. Princeton Math. Ser., 38. Princeton Univ. Press, Princeton, NJ, 1989.

	\bibitem{LiZhang2025} C. Li and B. Zhang, Covering instability for the existence of positive scalar curvature metrics. {\it Preprint}, 2025, arXiv:2506.13885v2.
	
	\bibitem{Lich1963} A. Lichnerowicz, Spineurs harmoniques. {\it C. R. Acad. Sci. Paris} 257 (1963), 7--9.

	\bibitem{Nicolaescu2007} L. I. Nicolaescu, On the space of Fredholm operators. {\it An. \c{S}tiint. Univ. Al. I. Cuza Ia\c{s}i. Mat. (N.S.)} 53 (2007), no. 2, 209--227.

	\bibitem{Rade2023} D. R{\"a}de, Scalar and mean curvature comparison via $\mu$-bubbles. {\it Calc. Var. Partial Differential Equations} 62 (2023), no. 7, 187, 39 pages.

    \bibitem{Rade1994} J. R{\aa}de, Callias' index theorem, elliptic boundary conditions, and cutting and gluing. {\it Commun. Math. Phys.} 161 (1994), 51--61.

	\bibitem{RS2014} J. Rosenberg and S. Stolz, Metrics of positive scalar curvature and connections with surgery. {\it Surveys on Surgery Theory}, Vol. 2, eds. by S. Cappell, A. Ranicki and J. Rosenberg. Princeton Univ. Press, Princeton, NJ, 2014, 353--386.
	
	\bibitem{Rotman} J. J. Rotman, {\it An Introduction to Homological Algebra}. Universitext. Springer New York, 2008.
	
	\bibitem{ODE} W. Walter, {\it Ordinary Differential Equations}. Grad. Texts in Math. 182. Springer New York, 1998.
	
	\bibitem{Yu2024} J. Wang, Z. Xie, and G. Yu, A proof of Gromov's cube inequality on scalar curvature. {\it J. Differential Geom.} 128 (2024), 1285--1300.

	\bibitem{Zeid2020} R. Zeidler, Width, largeness and index theory. {\it SIGMA} 16 (2020), 127, 15 pages.
	
	\bibitem{Zeid2022} R. Zeidler, Band width estimates via the Dirac operator. {\it J. Differential Geom.} 122 (2022), 155--183.

	\bibitem{Zhu2021} J. Zhu, Width estimate and doubly warped product. {\it Trans. Amer. Math. Soc.} 374 (2021), no. 2, 1497--1511.
\end{thebibliography}
\end{document}